\documentclass[12pt,reqno]{amsart}

\usepackage{graphicx}

\usepackage{amssymb}
\usepackage{amsthm}

\usepackage{amsfonts}
\usepackage{amsmath}
\usepackage{amstext}
\usepackage{tikz}
\usepackage{mathrsfs}

\theoremstyle{plain}
\newtheorem{thm}{Theorem}[section]
\newtheorem*{thmB}{Theorem B}

\newtheorem*{thmA}{Theorem A}

\newtheorem*{vlem}{The Vertex Lemma}
\newtheorem*{Elem}{The Euler Lemma}

\newtheorem*{total}{Theorem C}

\newtheorem*{thmE}{Theorem E}

\newtheorem*{thmD}{Theorem D}
\newtheorem{cor}[thm]{Corollary}
\newtheorem{lem}[thm]{Lemma}

\newtheorem{lem/defi}[thm]{Lemma/Definition}
\newtheorem{defi/lem}[thm]{Definition/Lemma}
\newtheorem{prop/defi}[thm]{Proposition/Definition}

\theoremstyle{definition}
\newtheorem{defi}[thm]{Definition}

\theoremstyle{remark}
\newtheorem{rem}[thm]{Remark}

\newtheorem*{attention}{Attention}
\newtheorem{example}[thm]{Example}

\newtheorem*{starex}{Example}
\newtheorem{nota}[thm]{Notations}

\newtheorem*{convention}{Convention}

\numberwithin{equation}{section}

\newcommand{\C}{\mathbb{C}}

\newcommand{\GC}{\mathcal{GC}}

\newcommand{\SU}{\mathcal{S}}
\newcommand{\I}{\mathcal{I}}

\newcommand{\PP}{\mathcal{NP}}

\newcommand{\R}{\mathbb{R}}

\newcommand{\F}{\mathbb{F}}

\newcommand{\gr}{\mathcal{G}}
\newcommand{\ga}{\gamma}
\newcommand{\al}{\alpha}
\newcommand{\oa}{\omega}

\newcommand{\Q}{\mathbb{Q}}
\newcommand{\Z}{\mathbb{Z}}

\newcommand{\s}{\smallskip}
\newcommand{\m}{\medskip}

\newcommand{\mcv}{\mathcal{V}}

\newcommand{\mer}{\mathscr M}

\newcommand{\ra}{\rightarrow}
\newcommand{\mt}{\mapsto}
\newcommand{\be}{\begin{equation}}
\newcommand{\ee}{\end{equation}}
\newcommand{\bee}{\begin{equation*}}
\newcommand{\eee}{\end{equation*}}
\newcommand{\beqn}{\begin{eqnarray*}}
\newcommand{\eeqn}{\end{eqnarray*}}

\DeclareMathOperator{\Grad}{Grad}

\begin{document}

\begin{abstract}The level curves of an analytic function germ can have bumps (maxima of Gaussian curvature) at unexpected points near the singularity. This phenomenon is fully explored for $f(z,w)\in \C\{z,w\}$, using the Newton-Puiseux infinitesimals and the notion of gradient canyon. Equally unexpected is the Dirac phenomenon:  as $c\ra 0$, the total Gaussian curvature of $f=c$ accumulates in the minimal gradient canyons, and nowhere else.

Our approach mimics the introduction of polar coordinates in \textit{Analytic Geometry}.
\end{abstract}
\title {A'Campo Curvature Bumps and the Dirac Phenomenon\\ Near A Singular Point}

\author{Satoshi Koike, Tzee-Char Kuo and Laurentiu Paunescu}

\address{Department of Mathematics, Hyogo University of Teacher Education, Hyogo,
Japan; School of Mathematics and Statistics, University of Sydney,
  Sydney, NSW, 2006, Australia.}
\email{koike@hyogo-u.ac.jp; tzeechar@gmail.com; laurent@maths.usyd.edu.au}

\date{\today}



\maketitle

\vspace{0 truecm}

\section {Introduction}\label{intro}

Let us first expose the idea in the real case.  Consider  a real analytic function germ $f(x,y)$, $f(0,0)=f_x(0,0)=f_y(0,0)=0$. The level curves $f=c$, $0<|c|<\epsilon$, have ``bumps" near $0$. This is illustrated in the following two examples of cusps, shown in Fig.\ref{fig:cusp}:
\begin{eqnarray*}\label{1a}& f_2(x,y)=\frac{1}{2}x^2-\frac{1}{3}y^3,\quad  f_4(x,y)=\frac{1}{4}x^4-\frac{1}{5}y^5.&\end{eqnarray*}

Each level curve  $f_2=c$ attains maximum curvature (bump) when crossing the $y$-axis,  along which the curvature tends to infinity as $y$ tends to zero.

\m

A profound discovery, which one of us (Koike) learned in a lecture of A'Campo at Angers in $2000$, is that bumps of such cusps can appear along peculiar arcs, not the $y$-axis.

For example, the curvature of $f_4=c$ is actually $0$ on the $y$-axis; the maximum is attained instead as the level curve crosses $x=\pm ay^{4/3}+\cdots$, $a=(2/7)^{1/6}$.  Following \cite{langevin} we attribute this discovery to Nobert A'Campo, whence the name ``\textit{A'Campo curvature bumps}".

\begin{figure}[htb]
 \centering
 \includegraphics[width=.3\linewidth]{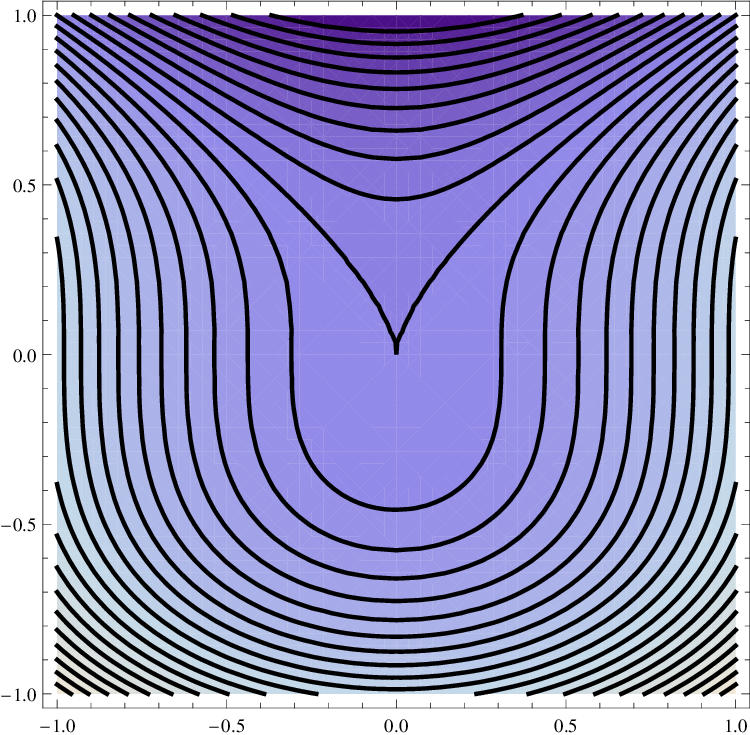}
 \hfil
 \includegraphics[width=.3\linewidth]{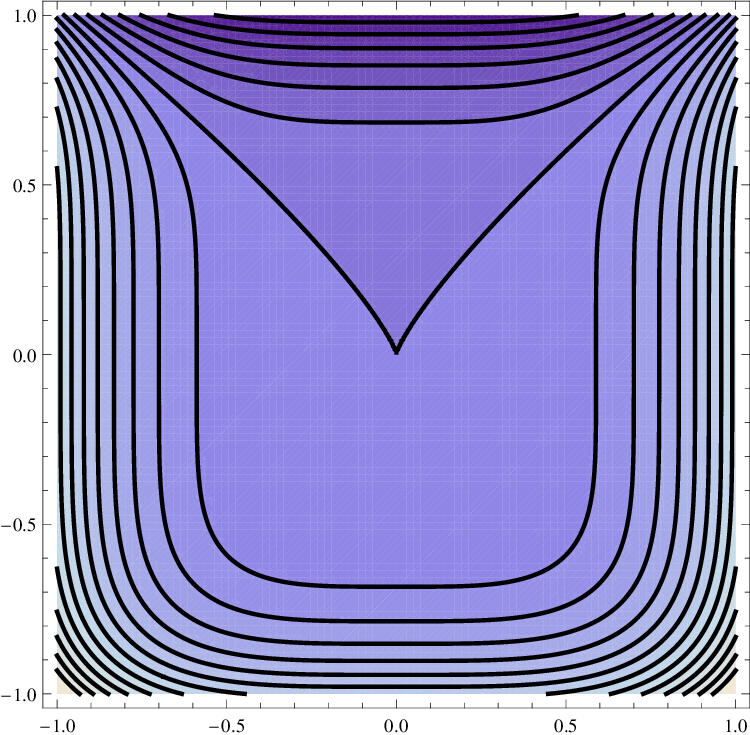}
 \caption{$f_2 =c\,, \, f_4=c$}
 \label{fig:cusp}
 \end{figure}
 
 \m

In this paper, we explore this phenomenon in the complex case, using the \textit{Newton-Puiseux infinitesimals} (defined in \cite{kuo-pau}, \cite{kuo-pau2}, recalled below) and the notion of ``\textit{gradient canyon}'' (defined in \S\ref{structures}). (The real case, with some subtle differences, is studied in \cite{kkp1}, where, for instance, Proposition 5.7 shows a striking difference.) Our Theorems C and D are, in a way, a continuation of Langevin \cite{langevin} (total curvature), Garcia Barosso and Teissier \cite{barosso} (concentration of curvature). The work of Siersma and Tibar \cite{S-T} is in a different direction.

\m

Take $f(z,w)\in \C\{z,w\}$. A level curve $\SU_c:f(z,w)=c$ is a Riemann Surface in 
$\C^2$, having Gaussian curvature (at regular points)
\be\label{1b}
K(z,w)=\frac{2|\Delta_f(z,w)|^2}{(|f_z|^2 +|f_w|^2)^3},\quad \Delta_f(z,w)\!:=\begin{vmatrix}{f_{zz}}&{f_{zw}}&{f_z}\\
{f_{wz}}&{f_{ww}}&{f_w}\\{f_z}&{f_w}&0\end{vmatrix}.
\ee
 
This is actually the \textit{negative} of the usual Gaussian curvature defined in text books (\cite{docar}).
The formula is due to Ness (\cite{ness}); for convenience we give a proof in 
\S\ref{appendix}\,(I).

\m

Consider a holomorphic map germ with parameter $t$,
\be\label{holo}\al: (\C,0)\longrightarrow (\C^2,0), \quad \al(t)=(z(t),w(t))\not \equiv 0.
\ee

The image \textit{set germ} $\al_*\!:=Im(\al)$ is an \textit{irreducible curve germ} at $0$ in $\C^2$, also called a \textit{holomorphic arc} at $0$. It has a unique (complex) tangent line $T(\al_*)$ at $0$, $T(\al_*)\in\C P^1$.

\m 

We call $\al_*$ a \textit{Newton-Puiseux infinitesimal} at $T(\al_*)$. 

\m

The \textit{Enriched Riemann Sphere} (``enriched" with infinitesimals) is, by definition, 
$$\C P_{enriched}^1\!:=\{\al_*\},\; \text{the set of holomorphic arcs at $0\in \C^2$},$$
furnished with a number of structures defined in \S\ref{structures}.

\m

The image of $t\mt (at,bt)$ is identified with $[a:b]\in \C P^1$, so that $\C P^1\subset \C P_{enriched}^1$.

\m

Let $f(z,w)\in \C\{z,w\}$ be given. Take $\al_*$, and  a parametrization $\al(t)$. 

\m

Suppose $K(\al(t))\not \equiv 0$. Then there exists a \textit{unique} pair $(a,L)$ such that
\be\label{Kk}
\lim_{t\ra 0}\frac{K(\al(t))}{\|\al(t)\|^L}=a,\quad a>0, \; L\in \Q\;\,(\text{rationals}),
\ee
where $(a,L)$ \textit{depends on} $\al_*$, but \textit{not} on the choice of $\al(t)$. This is shown in \S \ref{appendix}\,(II). 

We can write (\ref{Kk}) as
\bee
K(\al(t))\sim a\|\al(t)\|^L,
\eee
where $A(t) \sim B(t)$ means $\lim_{t\ra 0} A(t)/B(t)=1.$ Accordingly, we define the ``\textit{value space}"
\bee 
\mcv\!:=\{a\delta^L\,|\,a\ne 0, \; L\in \Q\}\cup \{O_\mcv\},
\eee 
where $\delta$ is a symbol,
\bee 
a\delta^L\!:=(a,L),\quad O_\mcv\!:=0\delta^\infty\!:=(0,\infty);
\eee
the latter is added so that when $K\equiv 0$ along $\al_*$, the value is $O_\mcv$, the ``\textit{origin}" of $\mcv$.

A \textit{lexicographic ordering} on $\mcv$ is defined: $O_\nu$ is the smallest
element,
\be\label{order} 
a\delta^L>a'\delta^{L'}\;\,\textit{if and only if\, either}\;\, L<L',\;\, \textit{or else}\;\,
L=L',\;a>a'.
\ee 
We also define $a\delta^L\gg a'\delta^{L'}$ (\textit{substantially larger than}) if $ L<L'$. Thus, for example, $$2\delta^{3/2}>\delta^{3/2}\gg 10^{10}\delta^2\gg O_\mcv.$$

Let $f(z,w)$ be given. Using (\ref{Kk}), we define the \textit{curvature function}
\be\label{KL} 
K_*: \C P_{enriched}^1\longrightarrow \mcv,\quad \alpha_*\mt a\delta^L,
\ee
and also the exponent function
\be\label{LEF}L_*:\C P_{enriched}^1\longrightarrow \Q\cup \{\infty\},\quad \al_*\mt L.
\ee

What we do in this paper mimics the shifting from Cartesian coordinates to the polar coordinates in \textit{Analytic Geometry}, where, for example, in order to study 
\bee x^2-y^2=(x^2+y^2)^2, \quad (x,y)\in \R^2,\quad \text{(the ``figure $\infty$" curve,)}
\eee
we ``lift" it to
$r^2=\cos 2\theta$ in the $(r,\theta)$-space.

\m

We proceed similarly, in order to understand the behaviour of $K(z,w)$ as $(z,w)\rightarrow 0$ in $\C^2$, we study the function $K_*$ on $\C P_{enriched}^1$. This turns out to be quite effective and convenient, even though $\C P^1_{enriched}$, being a space of arcs, is infinite dimensional.

\m

In \S\ref{structures}, the main concepts, the A'Campo bump, the gradient canyon, the infinitesimal disc and line, etc., are introduced, and  the main results, Theorem A to Theorem E, are stated. 

\m 

In Theorem A, supplemented by Theorem B, we show how to compute the A'Campo bumps. Examples are given to expose the arcs in $\C^2$ along which the bumps appear. 

\m

We then study the total Gaussian curvature. 
In Theorem C, the integral of the curvature over a gradient canyon, as $c\ra 0$, is given 
by a Gauss-Bonnet type formula.

\m

Theorem D asserts that the curvature on an infinitesimal line \textit{concentrates} in the gradient canyons contained therein, and \textit{nowhere else}; in other words, the integral of the curvature, as $c\rightarrow 0$, behaves like that of a Dirac function. This is a much more precise information than that given in 
\cite{langevin}, \cite{barosso}. 

\m

The integral extends to a measure on $\C_{enriched}$, the $\sigma$-algebra of measurable sets being generated by the enriched discs (Remark\,\ref{RMC}, Remark\,\ref{LRK}).
\m

In Theorem E the polars are perturbed within the canyons to create twin networks of iterated torus knots; the Milnor number $\mu_f$ is expressed as their linking number.

\m  

Throughout this paper we assume $f(0,0)=0 \,\, \text{and }f(z,w)$ is \textit{not} of the form
\be\label{KCONST}f(z,w)=unit \cdot[z-\zeta(w)]^m,
\ee
where $\zeta(w)$ is an integral power series. Indeed, if (\ref{KCONST}) holds,  then $f(z,w)^{1/m}$ is holomorphic and non-singular, having the same Gaussian curvature $K(z,w)$ as $f(z,w)$.

We do not assume $0$ is an isolated singularity of $f(z,w)$.

Theorems A, B, C have been announced in \cite{kkp}.

\m

We are  grateful to the referees, their constructive criticism has led to a much better presentation of this paper. We also thank Don Cartwright for his advice on Carath\'eodory's  Extension Theorem.

\section{The Space $\C P_{enriched}^1$ and the Main Theorems}\label{structures}

The classical Newton-Puiseux Theorem asserts that the field $\F$
of convergent fractional power series in an indeterminate $y$ is
algebraically closed. (\cite{walker}, \cite{wall}.)

Recall that a non-zero element of $\F$ is a (finite or infinite)
convergent series
\bee
\alpha(y)=a_0y^{n_0/N}+\cdots +a_iy^{n_i/N}+\cdots,\quad
n_0<n_1<\cdots,
\eee
where $0\ne a_i\in \C$, $n_i \in \Z$ (integers), $N\in \Z^+$ (positive integers), and 
$$GCD(N, n_0, n_1, ...)=1,\quad  \limsup|a_k|^{\frac{1}{n_k}}<\infty.
$$

 We shall also write $\al(y)$ as $\al$.  The \textit{conjugates} of $\alpha$ are
\bee \alpha_{conj}^{(k)}(y)\!:=\sum a_i \theta^{kn_i}y^{n_i/N},\quad
0\leq k\leq N-1,\quad \theta\!:=e^\frac{2\pi \sqrt{-1}}{N}.
\eee
The \textit{order} of $\al$ is 
\begin{eqnarray*}& O(\al)\!:=O_y(\alpha)\!:=\frac{n_0}{N}\;\,\text{if}\;\, \al\ne 0; \quad O(\al)\!:=\infty \;\, \text{if}\;\, \al=0.&
\end{eqnarray*}
 
The \textit{Puiseux multiplicity} of $\al$ is, by definition, 
$$m_{puiseux}(\al)\!:=N.$$

Let $e\in \Q$, or $e=\infty$, be given. The $e$-\textit{jet} of $\al$ is, by definition,
\be\label{jet} 
J^{(e)}(\al)(y)\!:=\begin{cases}\al(y)\;\text{with all terms $y^q$, $q>e$, deleted},\quad & \text{if $e<\infty$},\\  \al(y) &\text{if $e=\infty$.}\end{cases}
\ee

\begin{example}The equation $z^2-y^3=0$ has a pair of conjugate roots $\pm y^{3/2}$,
$$\theta=-1, \quad O_y(\pm y^{3/2})=3/2, \quad m_{puiseux}(\pm y^{3/2})=2.$$

Next, consider the following polynomial equation of degree $4$ over $\F$:
\bee
(z^2-y^3)^2-zy^5=0.
\eee
Using the Newton Polygon method (\cite{walker}), the roots are found to be 
\begin{eqnarray}&\label{beta}y^{\frac{3}{2}}\pm \frac{1}{2}y^{\frac{7}{4}}+\cdots, \; -y^{\frac{3}{2}}\pm \frac{\sqrt{-1}}{2}y^{\frac{7}{4}}+\cdots,&
\end{eqnarray}
where all four are conjugates, $\theta=\sqrt{-1}$. Let $\al$ denote any one of them, then
\begin{eqnarray*}& m_{puiseux}(\al)=4, \;\, O_y(\al)=\frac{3}{2}, \;\, J^{(3/2)}(\al)=\pm y^{\frac{3}{2}}, \;\, m_{puiseux}(J^{(3/2)}(\al))=2.&
\end{eqnarray*}
\end{example}

Now, $\C P_{enriched}^1$ is the union of two \textit{``enriched complex lines"}
$$\C P_{enriched}^1=\C_{enriched}\cup \C_{enriched}'$$
as in \textit{Projective Geometry}, where
\bee\C_{enriched}\!:=\{\beta_*\,|\,T(\beta_*)\ne [1:0]\},\quad \C_{enriched}'\!:=\{\beta_*\,|\,T(\beta_*)\ne [0:1]\}.
\eee

Let us introduce coordinates on $\C_{enriched}$. Consider
\be\label{two}\F_0\!:=\{\al\in \F|O_y(\al)\geq 0\},\;\, \F_1\!:=\{\al|O_y(\al)\geq 1\}, \;\, \F_{1^+}\!:=\{\al|O_y(\al)>1\},
\ee
where $\F_0$ is an integral domain with quotient field $\F$; $\F_1$, $F_{1^+}$ are ideals of $\F_0$.

\m

Take any $\al\in \F_1$ in (\ref{two}). The map germ
\be
\label{para}
\al_{para}: (\C,0)\longrightarrow (\C^2,0), \quad t\mt (\al(t^N),t^N),\quad N\!:=m_{puiseux}(\al),
\ee
is holomorphic. Hence, as in (\ref{holo}), the holomorphic arc $\al_*\in \C_{enriched}$ is defined.

If $N$ is replaced by $kN$, $k\in \Z^+$, $\al_*$, being a \textit{set germ}, would not change.

\m

Of course, all conjugates of $\al$ lead to the same holomorphic arc $\al_*$. For example,
the roots in (\ref{beta}) give rise to the holomorphic arc with parametrisation $(t^6+\frac{1}{2}t^7+\cdots, t^4)$.

\m

The \textit{Newton-Puiseux
coordinate system} on $\C_{enriched}$ is, by definition, the surjection
\be\label{pi}\pi:\F_1\longrightarrow \C_{enriched}, \quad \al \mt \al_*.
\ee

When $\al_*$ is given, the conjugate class of $\al$ is unique. That is, $\pi$ induces a bijection between the set of conjugate classes in $\F_1$ and $\C_{enriched}$.

\begin{defi}\label{enrich} Let $\al_*\in \C_{enriched}$ be given. Take $e$, $1\leq e<\infty$, and $\rho\geq 0$. Let 
\be\label{ID}\mathcal{D}^{(e)}(\al_*;\rho)\!:=\{\beta_*\,|\,\beta(y)=[J^{(e)}(\al)(y)+cy^e]+\cdots ,\,\;|c|\leq \rho\},
\ee
where ``$\cdots$" means ``higher order terms", and
\be\label{IL}\begin{split}\mathcal{L}^{(e)}(\al_*)&\!:=\mathcal{D}^{(e)}(\al_*;\infty)\!:=\cup_{0<\rho<\infty}\mathcal{D}^{(e)}(\al_*;\rho)\\&=
\{\beta_*|\beta(y)=[J^{(e)}(\al)(y)+cy^e]+\cdots,\; |c|<\infty\}.\end{split}
\ee

In particular,
$$\mathcal{L}^{(1)}(\al_*)=\C_{enriched}, \quad \mathcal{D}^{(e)}(\al_*;0)\!:=\{\beta_*\,|\,\beta(y)=J^{(e)}(\al)(y)+\cdots\}.$$

We call $\mathcal{D}^{(e)}(\al_*,\rho)$ the \textit{enriched (complex) disc} along $\al_*$, of \textit{order} $e$, radius $\rho$, and $\mathcal{L}^{(e)}(\al_*)$ the \textit{enriched line} along $\al_*$ of \textit{order} $e$.

\m

\textit{When $e>1$, we also call $\mathcal{D}^{(e)}$ an \textit{infinitesimal} disc, and $\mathcal{L}^{(e)}$ an \textit{infinitesimal} line.}

\m

For convenience, we define 
\bee
\mathcal{D}^{(\infty)}(\al_*)\!:=\mathcal{L}^{(\infty)}(\al_*)\!:=\{\al_*\}, \;\text{a singleton}.
\eee 

Finally, for $1\leq e\leq \infty$, using (\ref{jet}), we define the $e$-\textit{jet map}:
\bee
J^{(e)}: \C_{enriched}\to \C_{enriched}, \quad \al_*\mapsto J^{(e)}(\al_*)\!:=J^{(e)}(\al)_*.
\eee
\end{defi}

\m

When $\al_*$ is given, the right-hand side of (\ref{ID}) \textit{does not depend} on the choices of $\al$ in the conjugate class for $\al_*$. We can use any $\al$ in the conjugate class. Note also that
$$J^{(e)}(\al_*)=J^{(e)}(\tilde{\al}_*)\Longleftrightarrow \mathcal{D}^{(e)}(\al_*; \rho)=\mathcal{D}^{(e)}(\tilde{\al}_*; \rho),
$$and that
\bee O_y(\al-\tilde{\al})\geq e\Longleftrightarrow  \mathcal{L}^{(e)}(\al_*)=\mathcal{L}^{(e)}(\tilde{\al}_*).
\eee

Two enriched lines are either disjoint, or else one contains the other (or are equal),
 \be\label{TWO}\mathcal{L}^{(e)}(\al_*) \subsetneq \mathcal{L}^{(\tilde{e})}(\tilde{\al}_*)\Longleftrightarrow e>\tilde{e}\;\text{and} \;  O_y(\al-\tilde{\al})\geq \tilde{e}.
 \ee

Let $f(z,w)$ be given. For convenience, let us apply a generic unitary transformation, if necessary, to bring $f(z,w)$ to the form
\be\label{mini} f(z,w)\!:=H_m(z,w)+H_{m+1}(z,w)+\cdots,\;\, H_m(1,0)\ne 0,\;\, O(f)=m,
\ee 
where $H_k$ is a homogeneous  $k$-form.  (When $H_m(1,0)\ne 0$, we say $f$ is \textit{mini-regular} in $z$.)

A unitary transformation preserves the metric, hence also the Gaussian curvature.

\m 

From now on we shall restrict our attention to $\C_{enriched}$.

\begin{defi}\label{DAB}Given $\beta_*\in \C_{enriched},$ we say $\mathcal{D}^{(d)}(\beta_*; 0)$ is an \textit{A'Campo bump} of $K_*$, of \textit{degree} $d$, $1\leq d <\infty$, if the following three conditions are satisfied. 
\begin{enumerate}
  \item The curvature function $K_*$ is constant on $\mathcal{D}^{(d)}(\beta_*;0)$. That is, 
  \bee J^{(d)}(\xi_*)=J^{(d)}(\beta_*) \implies  K_*(\xi_*)= K_*(\beta_*);
  \eee
  \item If $\rho>0$ is sufficiently small, then 
  \bee J^{(d)}(\xi_*)\in \mathcal{D}^{(d)}(\beta_*; \rho)\implies K_*(\xi_*)\leq K_*(\beta_*);
  \eee
  where ``$<$" holds for \textit{generic} $\xi_*$, i.e., if 
$a$ is generic, $|a|\leq \rho$, then
  \bee
  \xi(y)=[J^{(d)}(\beta)(y)+ay^d]+\cdots\implies K_*(\xi_*)< K_*(\beta_*);
  \eee
  \item In the case $d>1$, there exists $\epsilon>0$ such that
  \bee J^{(d-\epsilon)}(\xi_*)= J^{(d-\epsilon)}(\beta_*),\;J^{(d)}(\xi_*)\ne J^{(d)}(\beta_*)\implies K_*(\xi_*)\ll K_*(\beta_*).
  \eee
\end{enumerate}
\end{defi}

Now, let $f(z,w)$ be as in (\ref{mini}). 
Consider the Newton-Puiseux factorizations:
\begin{eqnarray}\label{ff}
f(z,y)=unit\cdot \prod _{i=1}^m(z-\zeta_i(y)), \quad f_z(z,y)=unit\cdot \prod_{j=1}^{m-1}(z-\ga_j(y)),
\end{eqnarray} 
where $\zeta_i$, $\ga_j\in \F_1$.
Each $\ga_j$, and $\ga_{j*}$, is called a \textit{polar}. 

\m

If a polar $\ga$ is also a root of $f$, i.e. $f(\ga(y),y)\equiv 0$, then it is a multiple root of $f$. We shall see in Theorem A that such polars do not give rise to A'Campo bumps. 

\m

If there are $k$ distinct $\zeta_i$ in (\ref{ff}), there are exactly $k-1$ polars (counting multiplicities) which are not roots of $f$. This is a consequence of Theorem 2.1 in \cite{kuo-par} (or Lemma 3.3 in \cite{kuo-lu}).
Since $f$ is not of the form (\ref{KCONST}), we have $k\geq 2$. Hence at least one such polar exists.

\m

The following implication follows from the Chain Rule in Calculus: For any $\al\in \F_1$,
\be\label{EMF}
f_z(\al(y),y)\equiv f_w(\al(y),y)\equiv 0 \implies f(\al(y),y)\equiv 0.
\ee

Now take a polar $\ga$, $f(\ga(y),y)\not \equiv 0$. By (\ref{EMF}), $\ga$ is not a common Newton-Puiseux root of $f_z$, $f_w$. Hence, if $q$ is sufficiently large, then
\be\label{GG}
O_y(\|\Grad\!f(\ga(y),y)\|)=O_y(\|\Grad\!f(\ga(y)+uy^q,y)\|), \; \forall\, u\in \C.
\ee

Let $d_{gr}(\ga)$ denote the \textit{smallest number} $q$ such that (\ref{GG}) holds for \textit{generic} $u\in \C$.

\m

In the case $f(\ga(y),y)\equiv 0$, write $d_{gr}(\ga)\!:=\infty$.

\begin{defi}Take a polar $\ga$, $d\!:=d_{gr}(\ga)$. The \textit{gradient canyon} of $\ga_*$ is, by definition,
\bee \GC(\ga_*)\!:=\mathcal{L}^{(d)}(\ga_*), \quad 1\leq d \leq \infty.
\eee

We call $d_{gr}(\ga_*)\!:=d_{gr}(\ga)$ the \textit{gradient degree} of $\ga$, and the \textit{degree} of $\GC(\ga_*)$.

We say $\mathcal{GC}(\ga_*)$ is \textit{minimal} if $d_{gr}(\ga_*)<\infty$ and for every polar $\ga_j$ with $d_{gr}(\ga_j)<\infty$,
\bee \mathcal{GC}(\ga_{j*}) \subseteq\mathcal{GC}(\ga_*) \Longrightarrow \mathcal{GC}(\ga_{j*})= \mathcal{GC}(\ga_*).
\eee
\end{defi}

(The gradient canyons are not topological invariants, see \S\ref{appendix}\,(III).)

\begin{example}\label{ex4}For $f=z^m-w^n$, $2\leq m\leq n$, $\ga=0$ is the only polar,
\begin{eqnarray*}&
\mathcal{GC}(0_*)=\{\al_*\,|\,\al(y)=uy^{d_{gr}(\ga)}+\cdots,\;u\in \C\}, \quad d_{gr}(\ga)=\frac{n-1}{m-1}.\end{eqnarray*}\end{example}

\begin{example}\label{example}Take $f=z^4-2z^2w^2-w^{100}$, having polars $\ga_1=0$, $\ga_2,\ga_3=\pm w$. We find
\bee  d_{gr}(\ga_1)=97,\quad  d_{gr}(\ga_2)=d_{gr}(\ga_3)=1.
\eee
Here $\mathcal{GC}(\ga_{1*})$ is \textit{minimal}, but $\mathcal{GC}(\ga_{2*})=\mathcal{GC}(\ga_{3*})=\C_{enriched}$ is not.

\m

Next consider $g=z^4-2z^2w^2$,  with $\ga_1=0$, $\ga_2$, $\ga_3=\pm w$. Then
\bee
\GC(\ga_{1*})=\{\ga_{1*}\}\subset \GC(\ga_{2*})=\GC(\ga_{3*})=\C_{enriched}. 
\eee
The latter is minimal, containing the former, $d_{gr}(\ga_{1*})=\infty$. 

\m

We shall see  in Theorem B that $\C_{enriched}$ is a minimal canyon only in very exceptional cases, and that no minimal canyon, except $\C_{enriched}$, can contain a singleton canyon.
\end{example} 

Now take a polar $\ga$, $d_{gr}(\ga)<\infty$. Define $L_\ga\in \Q$, and $R_\ga:\C\ra \C$, as follows.

\m 

First, we apply a unitary transformation, if necessary, so that $T(\ga_*)=[0:1]$,  $\ga\in \F_1$.

\m

For brevity let us write $d\!:=d_{gr}(\ga)$. If $d>1$, define $L_\ga$ and $R_\ga(u)$ by the equation
\be\label{Rga}
K(\ga(y)+uy^d,y)=2R_\ga(u)y^{2L_\ga}+\cdots,\quad R_\ga(u)\not \equiv 0.
\ee
Note that for \textit{generic} $u\in \C$,
\be\label{ugg}
L_\ga=\frac{1}{2}L_*((\ga+uy^d)_*), \quad K_*((\ga+uy^d)_*)=(2R_\ga(u), L_*(\ga)),
\ee
where $L_*$ was defined in ({\ref{LEF}}).

In the case $d=1$, we define $L_\ga$ and $R_\ga(u)$ by
\begin{eqnarray}\label{Rgaa}&K(\ga(y)+\frac{uy}{\sqrt{1+|u|^2}},\frac{y}{\sqrt{1+|u|^2}})=2R_\ga(u)y^{2L_\ga}+\cdots,\quad R_\ga(u)\not \equiv 0.&
\end{eqnarray}
Since $\ga(y)$ has no linear term, (\ref{ugg}) remains true.

\m

(We can use (\ref{Rgaa}) also for the case $d>1$; but (\ref{Rga}) is easier for computation.)

\begin{lem}\label{LemA}The function $R_\ga(u)\,(\not\equiv 0)$ is defined and continuous for all $u\in \C$,
\begin{eqnarray}&\label{L} R_\ga(u)\geq 0, \;\, \lim_{u\ra \infty}R_\ga(u)=0; \;\, L_\ga=-d_{gr}(\ga).\end{eqnarray}

In particular, $R_\ga(u)$ is bounded, hence has at least one local maximum.
\end{lem}

The absolute maximum is of course attained. This lemma will be proved in 
\S\ref{proofs}.

\begin{thmA}Take a minimal gradient canyon $\mathcal{GC}(\ga_*)$. Take $c\in \C$ such that $R_\ga(c)$ is a local maximum of $R_\ga(u)$ and let  
\bee
\beta(y)\!:=[J^{(d)}(\ga)(y)+cy^d]+\cdots, \quad d\!:=d_{gr}(\ga_*).
\eee
Then $\mathcal{D}^{(d)}(\beta_*; 0)$ is an A'Campo bump. 
All A'Campo bumps are of this form.
\end{thmA}

Since $f(z,w)$ is not of the form (\ref{KCONST}), $d_{gr}(\ga)<\infty$ for at least one polar $\ga$. There exists at least one minimal gradient canyon, hence at least one A'Campo bump.
 
\begin{starex}\label{2q}For $f_2(z,w)=\frac{1}{2}z^2-\frac{1}{3}w^3$, there is only one polar $\ga=0$, having $d_{gr}(\ga)=2$,
$$R_\ga(u)=(|u|^2+1)^{-3},\;\, L_\ga=-2,\;\, K(\ga+uy^2,y)=2R_\ga(u)y^{-4}+\cdots.$$
Hence $R_\ga(u)$ is maximum at $u=0$. The A'Campo bump is $\mathcal{D}^{(2)}(0_*;0)$.

Next, consider $f_4(z,w)=\frac{1}{4}z^4-\frac{1}{5}w^5$, having only one polar $\ga=0$, 
\begin{eqnarray*}& d\!:=d_{gr}(\ga)=\frac{4}{3}, \;\Delta=z^2w^3(4z^4-3w^5),\;  R_\ga(u)=\frac{9|u|^4}{(|u|^6+1)^3},\; L_\ga=-d.&\end{eqnarray*}
Here $R_\ga(u)$ is maximum on the circle $|u|=(2/7)^{1/6}$, like a volcanic ring; each $u$ on the ring leads to an A'Campo bump. This infinite family of bumps shows up in Fig.\,\ref{fig:cusp} as two real arcs. (However, $R_\ga(0)=0$, $\mathcal{D}^{(e)}(0_*;0)$ is not a bump;  in Fig.\,\ref{fig:cusp}, $K=0$ along the $y$-axis.)
\end{starex}

Now let us factor the initial form $H_m(z,w)$ in (\ref{mini}):
\be\label{HG}H_m(z,w)=c(z-z_1w)^{m_1}\cdots (z-z_rw)^{m_r},\;\, m_i\geq 1, \;\, z_i\ne z_j \;\,\text{if}\;\,i\ne j,
\ee 
and $1\leq r\leq m$, $m=m_1+\cdots +m_r$, $c\ne 0$. 

\m

If $r=m$, all $m_i=1$. If $r<m$ then $H_m(z,w)$ is \textit{degenerate}, and vice versa.

\begin{thmB}\label{ADD1}Every gradient canyon of degree $d$, $1<d<\infty$, is minimal; gradient canyons with $1<d_{gr}\leq \infty$ (singleton canyons included) are mutually disjoint. 
 
There are exactly $r-1$ polars of gradient degree $1$ (counting multiplicities), they all have
 $\C_{enriched}$ as gradient canyon. 
 (If $r=1$, every gradient canyon of degree $<\infty$ is minimal.)

In the case $1<r\leq m$, $\C_{enriched}$ is minimal if and only if $f(z,w)$ has exactly $r$ distinct roots $\zeta_i$ in (\ref{ff}).
(In particular, if $H_m(z,w)$ is non-degenerate, $\C_{enriched}$ is minimal.)
\end{thmB}

Theorems\,A and B will be proved in \S\ref{proofs}. They are used to find all A'Campo bumps.

Now we state Theorem\,C and Theorem\,D, to be proved in \S\ref{CD}.

\begin{defi}\label{disc}
Consider a given enriched disc of finite order and \textit{positive} radius
\bee
\mathcal{D}^{(e)}(\al_*;\rho), \quad 1\leq e< \infty, \; 0<\rho<\infty.
\eee
Take a compact disc in $\C^2$:
\bee D_{isc}(0;\eta)\!:=\{(z,w)\in \C^2\,|\,\sqrt{|z|^2+|w|^2}\leq\eta\},\;\eta>0\;\,\text{sufficiently small}.
\eee

The \textit{horn domain} (in $\C^2$) associated to $\mathcal{D}^{(e)}(\al_*;\rho)$ is the compact subset of $\C^2$:
\be\label{hd} H_{orn}^{(e)}(\al_*;\rho;\eta)\!:
=\{(z,w)\in D_{isc}(0;\eta)\cap \beta_*|\beta(y)=J^{(e)}(\al)(y)+cy^{e}, |c|\leq \rho\}.
\ee

The \textit{total asymptotic Gaussian curvature} over $\mathcal{D}^{(e)}(\al_*;\rho)$ is, by definition,
\be\label{measure} \mer_f(\mathcal{D}^{(e)}(\al_*;\rho))\!:=\lim_{\eta\rightarrow 0}\{\lim_{c\rightarrow 0}\int_{\SU_c\cap H_{orn}^{(e)}(\al_*; \rho;\eta)}KdS\},\ee
where $\SU_c$ is the level surface $f=c$, $S$ the surface area, and $K$ the Gaussian curvature.

The \textit{total asymptotic Gaussian curvature} over an enriched line is, by definition,
\begin{eqnarray}&
\mer_f(\mathcal{L}^{(e)}(\al_*))\!:=\lim_{\rho\ra \infty}\mer_f(\mathcal{D}^{(e)}(\al_*;\rho)).&
\end{eqnarray}

The above definitions are easily extended to the case $e=\infty$, so that
\bee
\mer_f(\mathcal{D}^{(\infty)}(\al_*))=\mer_f(\mathcal{L}^{(\infty)}(\al_*))=\mer_f(\{\al_*\})=0.
\eee
\end{defi}

\textit{Note}.  The order of integration in (\ref{measure}) is vital; if reversed, it is meaningless.

\m

The ``real picture" of a horn domain is shown in Fig.\,\ref{fig:disc}, as a model for $\mathcal{D}^{(e)}(\al_*;\rho)$.
\begin{figure}[htt]
   \centering
   \begin{tikzpicture}[scale=0.5]
 \path (2.7,2.3) node[above] {\tiny{$x=\rho y^{\frac{3}{2}}$}};
 \path (1.5,0) node[above] {\tiny{$\eta$}};
 \draw(0,-3) -- (0,0);
  \draw(-3,0) -- (3,0);
 
 \draw (-0.18,0.9) -- (0.18,0.9);
 \draw (-0.25,1.1) -- (0.25,1.1);
 \draw (-0.27,1.3) -- (0.27,1.3);
 \draw (-0.35,1.5) -- (0.35,1.5);
 \draw (-0.44,1.7) -- (0.45,1.7);
 \draw (-0.6,1.9) -- (0.6,1.9);
 \draw (-0.9,2.1) -- (0.9,2.1);
 \draw (-1.06,2.3) -- (1.06,2.3);
 \draw (-1.3,2.5) -- (1.3,2.5);
 \draw (-1.2,2.7) -- (1.2,2.7);
 
  \draw (3,0) arc (0:360:3);
  \draw (0,0) arc (180:120:3);   
   \draw (0,0) arc (0:60:3);
    \end{tikzpicture}
 \caption{Horn Domain $H_{orn}^{(3/2)}(0_*;\rho;\eta)$}\label{fig:disc}
 \end{figure}
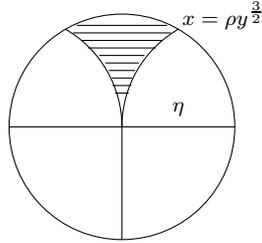
\begin{attention}In \textit{Analytic Geometry}, when returning to $(x,y)$ from $(r,\theta)$, there was no problem. Here, however, we have to be more careful when shifting from $\C_{enriched}$ back to $\C^2$. In (\ref{hd}), $J^{(e)}(\al)$ is a finite series, $\beta_*$ is defined \textit{globally}. If we use infinite series, the germs may not be defined in $D_{isc}(0;\eta)$. Even if  we take $J^{(e)}(\al)$, and all finite series $\beta$ such that 
$\beta_*\in \mathcal{D}^{(e)}(\al_*;\rho)$, then, since the terms of order $>e$ can have arbitrarily large coefficients, the union of all $\beta_*$ would cover the entire disc, except the $z$-axis. The integration (\ref{measure}) would be independent of $e$, hence meaningless. The picture in Fig.\,\ref{fig:disc} would not be horn-shaped.
\end{attention}

Now, a well-known formula to compute the Milnor number $\mu_f$ of $f(z,w)$ is
\begin{eqnarray}\label{MWK}
&\mu_f=\sum_{j=1}^{m-1}[O_y(f(\ga_j(y),y))-1], \quad \ga_j\; \text{as in (\ref{ff})}.&
\end{eqnarray}
In particular, $\mu_f=\infty$ if $0$ is not an isolated singularity.

\m 

We \textit{define} the \textit{Milnor number} of $f$ on a gradient canyon $\mathcal{GC}(\ga_*)$, $d_{gr}(\ga_*)<\infty$, to be
\begin{eqnarray}\label{M}&\mu_f(\mathcal{GC}(\ga_*))\!:=\sum_j[O_y(f(\ga_j(y),y))-1],&
\end{eqnarray}
the sum being taken over all $j$, $1\leq j\leq m-1$, such that 
$\ga_{j*}\in\mathcal{GC}(\ga_*)$.

The \textit{multiplicity} of $\GC(\ga_*)$ is, by definition,
\be\label{multi} m(\GC(\ga_*))\!:=\sharp \{j\,|\,1\leq j \leq m-1,\; \ga_{j*}\in\mathcal{GC}(\ga_*)\}.
\ee

Of course, combining (\ref{M}) and (\ref{multi}), we have
\begin{eqnarray*}&
\mu_f(\GC(\ga_*))+m(\GC(\ga_*))=\sum_jO_y(f(\ga_j(y),y)),&
\end{eqnarray*}
summing over those $j$ as in (\ref{M}).

\begin{example}\label{multi3}Take $f(z,w)=z^4-4zw^4-w^5$. The polars are $y^{4/3}$, $\omega y^{4/3}$ and $\omega^2y^{4/3}$, where $\omega^3=1$.
There is only one gradient canyon, having degree $4/3$. We have
\bee
 \mu_f(\GC(\ga_*))=3\times (5-1)=12, \quad m(\GC(\ga_*))=3.
\eee
\end{example}

\begin{total}{\upshape{(Compare \cite{langevin}, \cite{barosso}.)}} Let $\ga_*$ be a polar, $1< d_{gr}(\ga_*)\leq\infty$. Then
\be\label{ttl}\mer_f(\GC(\ga_*))=\begin{cases}2\pi[\mu_f(\GC(\ga_*))+m(\GC(\ga_*)],&\quad 1<d_{gr}(\ga_*)<\infty,\\0,&\quad d_{gr}(\ga_*)=\infty.\end{cases}
\ee
\end{total}

Thus, in Example\,\ref{multi3}, the total asymptotic Gaussian curvature is 
$$\mer_f(\GC(\ga_*))=2\pi[12+3]=30 \pi.$$

Next, for convenience, let us permute the indices, if necessary, so that 
\be\label{perm} \{\GC(\ga_{1*}),....,\;\GC(\ga_{s*})\}
\ee
is \textit{the set of all} minimal gradient canyons of \textit{degree} $>1$.
By (\ref{TWO}), they are mutually disjoint.

\begin{thmD}\label{ADDIRAC}{\upshape{(Compare  \cite{barosso}.)}}The ``Dirac phenomenon'' appears on every infinitesimal line. 

That is to say, let $\mathcal{L}\!:=\mathcal{L}^{(e)}$, $e>1$, be given, then
\bee
\mer_f(\mathcal{L})\ne 0 \Longleftrightarrow \GC(\ga_{j*})\subseteq \mathcal{L}\;\;\text{for some}\;\, j, \; 1\leq j \leq s;
\eee
and in this case
\begin{eqnarray}\label{DPP}&\mer_f(\mathcal{L})=\sum_{j}\mer_f(\GC(\ga_{j*})),
\end{eqnarray}
summing over all $j$ such that $\GC(\ga_{j*})\subseteq\mathcal{L}$, $1\leq j \leq s$.

For $e=1$, we have
\begin{eqnarray}\label{entire} &\mer_f(\C_{enriched})=2\pi m(r-1)+\sum_{j=1}^s\mer_f(\GC(\ga_{j*})),\;\, \text{$r$ as in (\ref{HG})}.&
\end{eqnarray}
In particular, if $f(z,w)$ has exactly $r$ distinct roots $\zeta_i$ in (\ref{ff}), 
then
\bee
\mer_f(\C_{enriched})=2\pi m(r-1).
\eee 
\end{thmD}

\begin{rem} \label{RMC}Using Carath\'{e}odory's Extension Theorem, we can extend $\mer_f$ to a \textit{measure} on the $\sigma$-algebra generated by the discs (see Remark \ref{LRK}). In particular,
\begin{eqnarray*}&\mer_f(\mathcal{D}^{(e)}(\al_*;0))\!:=\lim_{\rho\ra 0}\mer_f(\mathcal{D}^{(e)}(\al_*;\rho)) ,
\end{eqnarray*}
which need not be $0$, see (\ref{IPA}). Then (\ref{DPP}) asserts that the measure of $\mathcal{L}^{(e)}$ is \textit{concentrated} in the gradient canyons contained therein, whence referred to as the ``\textit{Dirac phenomenon}''.
\end{rem}

\begin{cor}{\upshape{(\cite{langevin}, \cite{barosso}.)}} If\, $0$ is an isolated singularity, then
\bee
\mer_f(\C_{enriched})=2\pi [\mu_f+O(f)-1].
\eee
\end{cor}

\m

\textit{We now state Theorem E}, to be proved in \S\ref{them}. Let us decompose $f_z$ in $\C\{z,w\}$:
 \bee f_z(z,w)=unit\cdot p_1(z,w)^{e_1}\cdots p_\tau(z,w)^{e_\tau},\quad e_k\geq 1,\eee where $ p_k(z,w)\in \C\{z,w\}$ are the (distinct) irreducible factors, mini-regular in $z$.

 For convenience, let us permute the indices of $\{\ga_j\}$, if necessary, so that
 $$p_k(\ga_k(w),w)=0,\quad 1\leq k\leq \tau.$$
 That is, $\ga_k$, together with the conjugates, are \textit{the} Newton-Puiseux roots of $p_k(z,w)$.

 \s

 Take a set of \textit{generic} numbers $\epsilon_{k,i}$:
 $$\epsilon\!:=\{\epsilon_{k,i}\,|\, 1\leq i \leq e_k,\;1\leq k\leq \tau,\;\,|\epsilon_{k,i}|\;\,\text{sufficiently small}\}.$$

 Take a fixed $k$. Let $\hat{\ga}_k(y)$ denote $\ga_k(y)$ with all terms $y^e$, $e>d_{gr}(\ga_k)$, deleted.
 We then perturb $\hat{\ga}_k(y)$ to
 \bee \ga^{(\epsilon)}_{k,i}(y)\!:=\hat{\ga}_k(y)+\epsilon_{k,i}y^{d_k}\in \gr(\ga_k),\quad d_k\!:=d_{gr}(\ga_k), \quad 1\leq i\leq e_k.\eee

 For each $i$, $\ga^{(\epsilon)}_{k,i}$ generates an \textit{irreducible} function germ
 $$p^{(\epsilon)}_{k,i}(z,w)\in \C\{z,w\}, \quad p^{(\epsilon)}_{k,i}(\ga^{(\epsilon)}_{k,i}(y),y)=0.$$ 
 We then define
 \begin{eqnarray*}&P^{(\epsilon)}_k(z,w)\!:=\prod_{i=1}^{e_k}p^{(\epsilon)}_{k,i}(z,w), \quad f^{(\epsilon)}_z(z,w)\!:=\prod_{k=1}^\tau P^{(\epsilon)}_k(z,w).\end{eqnarray*}

 \textit{The Newton-Puiseux roots of $f^{(\epsilon)}_z(z,w)$ are mutually distinct.} 

 \m

 Now let  $S^3\subset \C^2$ denote a $3$-sphere centered at $0$ with sufficiently small radius,
 $$V_{k,i}^{(\epsilon)}\!:=\{(z,w)\,|\,p^{(\epsilon)}_{k,i}(z,w)=0\}, \quad K_{k,i}^{(\epsilon)}\!:=V_{k,i}^{(\epsilon)}\cap S^3;$$ each $K_{k,i}^{(\epsilon)}$ is an iterated torus knot (\cite{EN}, \cite{milnor}, \cite{reeve}). We call
 $$\mathscr{C}_k^{(\epsilon)}\!:=\{K_{k,i}^{(\epsilon)}|1\leq i\leq e_k\},\quad \mathscr{N}^{(\epsilon)}\!:=\{\mathscr{C}_k^{(\epsilon)}|1\leq k\leq \tau\},$$ a \textit{cable} of knots, and a \textit{network} of cables, respectively. 
 
\m

 Next, take \textit{another} set of generic numbers $\delta\!:=\{\delta_{k,i}\}$. Define $p_{k,i}^{(\delta)}$, $f_z^{(\delta)}$, $K_{k,i}^{(\delta)}$, $\mathscr{C}_k^{(\delta)}$, $\mathscr{N}^{(\delta)}$. 
 
 \m

 Let $\mathscr{L}(K_{k,i}^{(\epsilon)}, K_{l,j}^{(\delta)})$ denote the linking number of two knots in the usual sense. 
 We call \begin{eqnarray*}&\mathscr{L}(\mathscr{C}_k^{(\epsilon)},\mathscr{C}_l^{(\delta)})\!:=\sum_{i=1}^{e_k}\sum_{j=1}^{e_l}\mathscr{L}(K_{k,i}^{(\epsilon)}, K_{l,j}^{(\delta)}),&
 \end{eqnarray*}
 the \textit{linking number} of the cables.
 
 We call $\mathscr{N}^{(\epsilon)}$, $\mathscr{N}^{(\delta)}$ a \textit{twin} networks. Their \textit{linking number} is, by definition,
 \begin{eqnarray*}&
 \mathscr{L}(\mathscr{N}^{(\epsilon)},\mathscr{N}^{(\delta)})\!:=\sum_{k=1}^\tau\sum_{l=1}^\tau \mathscr{L}(\mathscr{C}_k^{(\epsilon)},\mathscr{C}_l^{(\delta)}).
 \end{eqnarray*}

\begin{thmE}Suppose $0$ is an isolated singularity of $f(z,w)$. Then
 \be \label{milnor}\mu_f=\mathscr{L}(\mathscr{N}^{(\epsilon)},\mathscr{N}^{(\delta)}).
 \ee
 \end{thmE}
 
  This formula provides a geometric interpretation of $\mu_f$ using gradient canyons.

\section{Newton Polygons Relative to An Arc}\label{NP}

Two lemmas are presented in this section. The first is used to compute $d_{gr}(\ga_*)$. The second is pivotal for proving of Lemma \ref{32}; the latter is pivotal for proving Theorem A.

\m

Let $\ga$ be a given polar in (\ref{ff}), $d_{gr}(\ga)<\infty$, i.e., $f_z(\ga(y),y)\equiv 0$, $f(\ga(y),y)\not \equiv 0$.

\m

We can apply a unitary transformation, if necessary, so that $\ga\in \F_{1^+}$, $T(\ga_*)=[0:1]$.

\m

We then change variables (formally):
\begin{equation}\label{3b}Z\!:=z-\ga(w),\quad W\!:=w,\quad
F(Z,W)\!:=f(Z+\ga(W),W).\end{equation}

Since $\ga\in \F_{1^+}$, i.e., $O(\ga)>1$, it is easy to see that
\be\label{D}\|\Grad_{z,w}\!f\|\sim\|\Grad_{Z,W}\!F\|,\quad
\Delta_f(z,w)=\Delta_F(Z,W)+\ga''(W)F_Z^3,\ee 
 and accordingly,
\be\label{KNF}
K(z,w)\sim\frac{2|\Delta_F+\ga''F_Z^3|^2}{(|F_Z|^2+|F_W|^2)^3},\quad O(\ga'')>-1.
\ee
(As before, $A\sim B$ means $A/B\rightarrow 1$ as $(z,w)\ra 0$.)

\m

The Newton polygon $\PP(F)$ is defined in the usual way, as follows. Let us write
$$F(Z,W)=\sum c_{iq}Z^iW^q,\quad c_{iq}\ne 0,\quad (i,q)\in \Z\times\Q.$$
A monomial term with $c_{iq}\ne 0$ is represented by a ``\textit{Newton dot}" at $(i,q)$. We shall simply call it a \textit{dot} of $F$. The \textit{boundary} of the convex hull generated by $\{(i+u,q+v)|u,\,v\geq 0\}$, for all dots $(i,q)$, is the Newton Polygon $\PP(F)$, having edges $E_i$ and angles $\theta_i$, as shown in Fig.\ref{fig:npF}. In particular, $E_0$ is the half-line $[m,\infty)$ on the $Z$-axis. (In \cite{kuo-lu}, \cite{kuo-par}, this is called the Newton Polygon of $f$ \textit{relative to} $\ga$, denoted by $\PP(f,\ga)$.)

For a line in $\R^2$ joining $(u,0)$ and $(0,v)$, let us call $v/u$ its \textit{co-slope}. Thus
$$\textit{co-slope of}\;\;E_s=\tan\theta_s.$$

Some elementary, but useful, facts are:
\begin{itemize}
\item If $i\geq 1$, then $(i,q)$ is a dot of $F$ \textit{if and only if}
$(i-1,q)$ is one of $F_Z$.
\item $F_Z(0,W)\equiv 0$; $F_Z$ has no dot of the form $(0,q)$; $F$ has no dot of the form $(1,q)$.
\item Since $f(\ga(w),w)\not\equiv 0$, we know $F(0,W)\not\equiv 0$. Let us write
\begin{equation} \label{3k}F(0,W)=aW^h+\cdots, \quad a\ne 0, \;\,
h\!:=O_W(F(0,W))\in \Q.\end{equation}
Then $(0,h)$ is a vertex of $\PP(F)$, $(0,h-1)$ is one of $\PP(F_W)$. (See Fig.\,\ref{fig:npZ}.)
\end{itemize}

\begin{figure}[htt]
 \begin{minipage}[b]{.4\linewidth}
   \centering
   \begin{tikzpicture}[scale=0.44]
   
   \draw(0,0) -- (0,11.5);
   \path (0,11.5) node[above] {\tiny{W}};
  
   \path (7.5,0) node[below] {\tiny{$m$}};   
   \draw(0,0) -- (8.1,0);
    \path (8,0) node[right] {\tiny{Z}};
     
  \draw[dotted] (0.5,8)--(2,5);
    
   \draw plot[mark=*] coordinates {(2,5) (2.7,4)};
   \draw plot[mark=*] coordinates {(0.5,8) (0,11)};
    
   \draw plot[mark=*] coordinates {(5.6,1.4) (7.6,0)};
    
    \draw[dotted] (3.6,3) -- (4.8,2);

 \draw[dotted] (1.5,4) -- (2.8,4);
  \path (2,3.7) node[above] {\tiny{$\theta_s$}};
 
   \path (6.3,-0.3) node[above] {\tiny{$\theta_1$}};
    
    \path (7.1,0.35) node[above] {\tiny{$E_1$}};
    
    \path (3.7,4.8) node[left] {\tiny{$E_s$}};
    
     \path (0.6,8) node[right] {\tiny{$(m_{top},q_{top})$}};
    
    \path (2.2,9.2) node[left] {\tiny{$E_{top}$}};
    
    \draw (6.8,0) arc (180:150:1);
    
    \end{tikzpicture}
   \caption{$\PP(F)$}\label{fig:npF}
    \end{minipage}
    \begin{minipage}[b]{.4\linewidth}
   \centering
 \begin{tikzpicture}[scale=0.4]
  
   \draw(0,0) -- (0,13);
   \path (0,13) node[above] {\tiny{$W$}};
    
    \draw[dotted](7,0.5) -- (10.5,0.5);
    
    \path (8.8,0.45) node[above] {$\theta_{top}$};
    
    \draw(0,0) -- (12,0);
    \path (12,0) node[below] {\tiny{$Z$}};
     
    \draw[dotted] (5,0) -- (5,11);
    
    \draw plot[mark=*] coordinates {(0,12) (9,2) (10.5,0.5)};
    
    \draw plot[mark=*] coordinates {(9,2) (7,5) (5.5,8)};
    
    \draw[dotted] (5,11) -- (9,2);
    
    \draw[dashed] (5.5,8) -- (5,11);
    
    \draw[dashed] (0.5,8) -- (0,11);
    
    \path (5.1,9.5) node[right] {\tiny{$L$}};
    
    \path (0.1,9) node[right] {\tiny{$L^*$}};
    
    \path (4.4,11) node[right] {$\circ$};
    
    \path (0,11) node {$\circ$};
    
    \path (0,12) node[right] {\tiny{$(0,h)$}};
    
    \path (0,10.9) node[left] {\tiny{$(0,h-1)$}};
    
    \path (9,2) node[right] {\tiny{${(\widehat{m}_{top},\widehat{q}_{top})}$}};
    
    \path (10.5,0.5) node[right] {\tiny{$(m_{top},q_{top})$}};
    
    \path (5.5,8) node[right] {\tiny{$(m^*+1,q^*)$}};
    
    \path (5.3,11) node[right] {\tiny{$(1,h-1)$}};
    
    \path (5,0) node[below] {$1$};
    
    \path (2.3,7) node[right] {\tiny{$E_{top}$}};
    
    \draw[dotted] (9.5,0.5) arc (180:145:1.3); 
     
    \end{tikzpicture}   
     
    \caption{$\PP(F)$ $vs$ $\PP(F_Z)$.} \label{fig:npZ}
    \end{minipage}
    \end{figure}
    
\begin{nota}\label{top}Let $E_{top}$ denote the edge whose left vertex is $(0,h)$, $h$ as in (\ref{3k}), and right vertex is $(m_{top}, q_{top})$, as shown in the figures. We call it the \textit{top} edge; the angle is $\theta_{top}$. 
\end{nota}
 
Let $(\widehat{m}_{top},\widehat{q}_{top})\ne (0,h)$ be the dot of $F$ on $E_{top}$ which is
\textit{closest} to $(0,h)$. (Of course, $(\widehat{m}_{top},\widehat{q}_{top})$ may coincide with $(m_{top}, q_{top})$.) Then, clearly,
\bee
\label{3d}2\leq \widehat{m}_{top}\leq m_{top},\quad \frac{h-\widehat{q}_{top}}{\widehat{m}_{top}}=\frac{h-q_{top}}{m_{top}}=\tan\theta_{top}.\eee     

Now we draw a line $L$ through $(1,h-1)$ with the following two properties:
\begin{itemize}
	\item If $(m',q')$ is a dot of $F_Z$, then $(m'+1,q')$ lies on or above $L$;
	\item There exists a dot $(m^*,q^*)$ of $F_Z$ such that $(m^*+1,q^*)\in L$. (Of course, $(m^*+1,q^*)$ may coincide with $(\widehat{m}_{top},\widehat{q}_{top})$.)
\end{itemize}

\begin{lem}\label{LH}Let $\sigma^*$ denote the co-slope of $L$. Then
\be\label{tan} (i)\;d_{gr}(\ga)=\sigma^*;\,(ii)\,\;\sigma^*\geq \tan\theta_{top};\;(iii)\;  \sigma^*=\tan\theta_{top}\Leftrightarrow \tan\theta_{top}=1.
\ee

All dots of $F_W$ lie on or above $L^*$, $L^*$ being the line through $(0,h-1)$ parallel to $L$.

 In the case $\tan\theta_{top}>1$, $(0,h-1)$ is the only dot of $F_W$ on $L^*$.
\end{lem}

In Example\,(\ref{ex4}), $L$ is the line joining $(1,n-1)$ and $(m,0)$, hence $d=\frac{n-1}{m-1}$.

\begin{starex}\label{ExampleGC}For $F(Z,W)=Z^4+Z^3W^{27}+Z^2W^{63}-W^{100}$ and
$\ga=0$, $\mathcal{NP}(F)$ has only two vertices $(4,0)$,
$(0, 100)$, while $\mathcal{NP}(F_Z)$ has three: $(3,0)$,
$(2,27)$, $(1,63)$. The latter two and $(0,99)$ are collinear, spanning $L^*$; $h=100$, $\sigma^*=(99-27)/2=36$.
\end{starex}

\begin{nota}\label{no}Take $e\geq 1$. Let $\oa(e)$ denote the weight system: $\oa(Z)=e$, $\oa(W)=1$.

Let $G(Z,W^{1/N})\in \C\{Z,W^{1/N}\}$ be given. Consider its weighted Taylor expansion relative to this weight. We shall denote the \textit{weighted initial form} by $\I_{\oa(e)}(G)(Z,W)$, or simply $\I_\oa(G)$ when there is no confusion.

If $\I_\oa(G)=\sum a_{ij}Z^iW^{j/N}$, the \textit{weighted order} of $G$ is $O_\oa(G)\!:=ie+\frac{j}{N}$.
\end{nota}

\begin{proof}Note that ($ii$) and ($iii$) are clearly true, since $(1,h-1)$ lies on or above $E_{top}$.

Next, if $(i,q)$ is a dot of $F_W$, then $(i,q+1)$ is one of $F$, lying on or above $E_{top}$. Hence, by ($ii$), all dots of $F_W$ lie on or above $L^*$.

It also follows that if $\tan\theta_{top}>1$, then $(0,h-1)$ is the only dot of $F_W$ on $L^*$. 

\m

Now we show $(i)$. 

It is easy to see that if $\tan\theta_{top}=1$, then $d_{gr}(\ga)=1$.

It remains to consider the case $\sigma^*>1$. Since $\ga$ is a polar, 
\be\label{LI}O_y(\|\Grad\!f(\ga(y),y)\|)=O_W(F_W(0,W))=h-1.
\ee

Let us first take weight $\oa\!:=\oa(e)$ where $e\geq \sigma^*$. In this case, since $\sigma^*>1$,
\bee\I_{\oa}(F_W)(Z,W)=ahW^{h-1},\quad O_\oa(F_Z)\geq h-1;\eee where $a$, $h$ are as in (\ref{3k}), $ah\ne 0$. Hence for all $u\in \C$,
$$ O_W(F_W(uW^e,W))=h-1,\quad O_W(F_Z(uW^e, W))\geq h-1.$$

It follows that $d_{gr}(\ga)\leq \sigma^*$. It remains to show that $\sigma^*>d_{gr}(\ga)$ is impossible.

\m

Let us take $\oa(e)$ with $e<\sigma^*$. Note that $(m^*,q^*)$ is a dot of $F_Z$ on $L^*$, where $(m^*+1,q^*)$ is shown in Fig.\ref{fig:npZ}.
Hence, for generic $u$,
$$O_W(F_Z(uW^e, W))<h-1,\quad O_W(\|\Grad\!F(uW^e,W)\|)<h-1.$$

Thus, by (\ref{LI}), we must have $d_{gr}(\ga)>e$. This completes the proof.
\end{proof}

For the next lemma, let us expand $\Delta_F$ in (\ref{D}):
\be\label{Del}\Delta_F=-F_{ZZ}F_W^2-F_Z^2F_{WW}+2F_ZF_WF_{ZW}.
\ee

\begin{lem}\label{L1}Let $\ga$ be a given polar, $1<d\!:=d_{gr}(\ga)<\infty$, $\oa\!:=\oa(e)$ a weight system.

If  $e\geq d$, then the following is true:
\be\label{3ee}
O_\oa(F_{ZZ}F_W^2)<\min\{O_\oa(F_Z^2F_{WW}), O_\oa(F_ZF_WF_{ZW}),O_\oa(\ga''F_Z^3)\}.\ee It follows that
\be\label{3e} \I_{\oa}(F_{ZZ}F_W^2)=\I_{\oa} (\Delta_F)=\I_{\oa}(\Delta_F+\ga''F_Z^3).
\ee

If we merely assume $e\geq \tan\theta_1$, then a weaker statement is true:
\be\label{3i} 
O_\oa(F_{ZZ}F_W^2)=O_\oa (\Delta_F), \;\, O_\oa(\Delta_F+\ga''F_Z^3)=\min\{O_\oa(\Delta_F), O_\oa(\ga''F_Z^3)\}.
\ee
\end{lem}
\begin{proof}
Since $F_Z$ has no dot on the $W$-axis, we can write
\be\label{write}\I_\oa(F_Z)(Z,W)\!:=a_kZ^kW^q+\cdots +a_MZ^MW^{q-(M-k)e}, \quad a_k\ne 0\ne a_M,\ee where $1\leq k\leq M$, $q\in \Q$.
Then, clearly, 
\begin{equation}\label{FFF}O_\oa(F_Z)=q+ke, \quad O_\oa(F_{ZZ})=q+(k-1)e.\end{equation}
 
Now suppose $e\geq \tan\theta_{top}$. Let $h$, $a$ be as in (\ref{3k}). In this case,
\be\label{hkj}O_\oa(F)=h,\quad O_\oa(F_W)=h-1, \quad q+(k+1)e \geq h.\ee The last inequality holds since $(k+1,q)$ lies on or above $E_{top}$.
Thus, by (\ref{FFF}), (\ref{hkj}),
\be\label{FOR}O_\oa(F_{ZZ}F_W^2)=q+(k-1)e+2(h-1),\;\,O_\oa(\ga'' F_Z^3)=[O(\ga)-2]+3(q+ke).
\ee 

We obtain the following:
\be\label{FOR2} O_\oa(F_{ZZ}F_W^2) \leq O_\oa(F_Z^2F_{WW})=O_\oa(F_ZF_WF_{ZW})=2(q+ke)+(h-2),
\ee
where ``$\leq$'' can be replaced by ``$<$'' if $e > \tan\theta_{top}$.

In the case $e\geq d$, by Lemma\,\ref{LH}, we have
\be\label{TTT} e>\tan\theta_{top}, \quad q+ke\geq h-1.\ee 

Note that $O(\ga'')>-1$. Hence, as an immediate consequence of (\ref{FOR}) and (\ref{TTT}),
$$O_\oa(F_{ZZ}F_W^2)<O_\oa(\ga'' F_Z^3).$$ 
This completes the proof of (\ref{3ee}), and hence also that of (\ref{3e}).

\s

We now prove (\ref{3i}). First suppose $e\geq \tan\theta_{top}$. 

In this case we do not necessarily  have the second inequality in (\ref{TTT}). However, the leading term of
\be\label{LET}\I_\oa(F_{ZZ} F_W^2)(Z,W)=ka_ka^2h^2Z^{k-1}W^{q+2(h-1)}+\cdots,\ee
has $Z$-order $k-1$. All other terms in $\I_\oa(\Delta_F)$ have $Z$-order $>k-1$. 

Hence no cancellation with the leading term of (\ref{LET}) can happen in $\Delta_F$, the first equality in (\ref{3i}) follows from (\ref{FOR2}). 

The second equality is also clear. The $Z$-order of $\I_\oa(\ga''F_Z^3)$ is $3k$, which is larger than that of $F_{ZZ}F_W^2$. Hence there is no cancellation, 
$$O_\oa(\Delta_F)=O_\oa(\ga''F_Z^3)\implies O_\oa(\Delta_F+\ga''F_Z^3)=O_\oa(\Delta_F).$$

On the other hand, the following is obviously true:
$$O_\oa(\Delta_F)\ne O_\oa(\ga''F_Z^3)\implies O_\oa(\Delta_F+\ga''F_Z^3)=\min\{O_\oa(\Delta_F), O_\oa(\ga''F_Z^3)\}.$$

We have proved the second equality in (\ref{3i}) in the case $e\geq \tan\theta_{top}$.

\m

Now suppose $\tan\theta_1\leq e<\tan\theta_{top}$. In this case
$$O_\oa(F_W)=q+(k+1)e-1,$$ where $q\geq 1$ since $(k+1,q)$ can be at worst the left vertex of $E_1$. Hence
$$O_\oa(F_{ZZ}F_W^2)=O_\oa(F_ZF_WF_{ZW})=3(q+ke)+e-2\leq O_\oa(F_Z^2F_{WW}).$$
(The last inequality is an equality if $q>1$.) The same $Z$-order argument proves (\ref{3i}).\end{proof}

\begin{starex}It can happen that $\ga''F_Z^3$ dominates $\Delta_F$. Take
$$f(z,w)\!:=(z-w^2)^m-w^n,\quad \ga=w^2,\quad e=n/m,$$
where $n>2m$. We then have 
\beqn &O_\oa(F_Z)=\frac{n}{m}\cdot(m-1),\; O_\oa(F_W)=n-1, \; O_\oa(\ga''F_Z^3)<O_\oa(F_{ZZ}F_W^2).&\eeqn\end{starex}

\section {The Lojasiewicz Exponent Function}\label{Loj}

Let $\al \in \F$ be given. Take $e\geq 1$. Take a generic $u\in \C$, or an indeterminate. Write 
\be\label{LD} \begin{split}
|\Delta_f(\al(y)+u y^e,y)|^2&\!:=N_{(\al,e)}(u)y^{2L_{\Delta}(\al,e)}+\cdots,\;
 N_{(\al,e)}(u)\not \equiv 0,\\
\|\Grad\!f(\al(y)+u y^e,y)\|^2&\!:=D_{(\al,e)}(u)y^{2L_{gr}(\al,e)}+\cdots,\, D_{(\al,e)}(u)\not \equiv 0,\end{split}\ee 
where $N_{(\al,e)}(u)$, $D_{(\al,e)}(u)$ are real-valued, non-negative, polynomials. We also write
\be\label{LAP}L_{\al}(e)\!:=L_{\Delta}(\al,e)-3L_{gr}(\al,e), \quad R_{(\al,e)}(u)\!:=N_{(\al,e)}(u)D_{(\al,e)}(u)^{-3}.\ee

Observe that $L_\Delta(\al,e)$, $L_{gr}(\al,e)$ are defined even if $e$ is \textit{irrational}. Hence they are \textit{piece-wise linear, continuous, increasing} functions of $e$, defined for all $e\in[1,\infty)$.

\m

As in Calculus, we say $\phi(x)$ is \textit{increasing} (resp.\,\textit{decreasing}, resp.\,\textit{strictly decreasing}) if 
$$x_1<x_2 \implies \phi(x_1)\leq \phi(x_2)\;(\text{resp.}\, \phi(x_1)\geq\phi(x_2),\;\text{resp.}\,\phi(x_1)> \phi(x_2)).
$$

We call $L_{\al}(e)$ the \textit{Lojasiewicz exponent function along} $\al$. It is \textit{piecewise linear, continuous} (but not necessarily increasing). Note that $K_*$, defined in (\ref{KL}), can be written as
\bee K_*(\eta_*)=2R_{(\al, e)}(u)\delta^{2L_\al(e)}, \quad \eta(y)\!:=\al(y)+uy^e, \; u\;\textit{generic}.
\eee

Note also that $R_\ga(u)$ and $L_\ga$ defined in (\ref{Rga}) are special cases: 
\be\label{RRR} R_\ga(u)=R_{(\ga,d)}(u),\quad L_\ga=L_\ga(d), \quad d\!:=d_{gr}(\ga).
\ee

\begin{lem}\label{32} Let $\ga$ be a given polar, $1<d\!:=d_{gr}(\ga)<\infty$. Then
\begin{enumerate}
\item $L_\ga(e)> -1$ for $1< e<\tan\theta_1$, $\theta_1$ being the first angle of $\PP(F)$ (Fig.\ref{fig:npF});
\item $L_\ga(e)$ is increasing for $e\in (\tan\theta_1,\tan\theta_{top})$;
\item $L_\ga(e)$ is decreasing for $e\in (\tan\theta_{top},d)$;
\item $L_\ga(e)$ is strictly decreasing for $e\in (d-\epsilon, d)$, $\epsilon>0$ sufficiently small;\item  $L_\ga(d)=-d=L_\ga$;\item If $O(\al-\ga)\geq d$, then $L_{gr}(\al,e)=L_{gr}(\ga,d)$ and $L_\al(e)$ is
increasing for $e\geq d$.
\end{enumerate}

The above open intervals can be replaced by closed intervals; $L_{\ga}(q)=-d$ is the absolute minimum of $L_{\ga}(e)$, $e\in[1, d]$; moreover, $L_\ga(q)\geq -d$ for $q>d$.
\end{lem}
\begin{proof} We first prove ($6$). Let $a$, $h$ be as in (\ref{3k}). 

If $O(\al-\ga)\geq d$, then for all $e\geq d$,
$$D_{(\al,e)}(u)\geq |ah|^2>0 \;\forall\; u,\;\, L_{gr}(\al,e)=L_{gr}(\ga,d)=h-1.$$
But $L_\Delta(\al,e)$ is increasing, hence so is $L_\al(e)$, $e\geq d$.

Let us apply (\ref{3e}) with $e=d$. Then ($5$) follows from:
$$O_{\oa(d)}(F_W)=h-1, \quad L_\Delta(\ga,d)=3(h-1)-d,\quad L_{gr}(\ga,d)=h-1.$$

To prove ($4$), take weight $\oa(d)$, and write expression (\ref{write}).
We then consider $\oa(d-\epsilon)$, where $\epsilon>0$ is \textit{sufficiently small}. We obviously have
$$\I_{\oa(d-\epsilon)}(F_Z)=a_MZ^M W^{q-(M-k)d}, \quad \I_{\oa(d-\epsilon)}(F_{ZZ})=a_M MZ^{M-1}W^{q-(M-k)d}.$$

Moreover, (\ref{3e}) remains true when $\oa(d)$ is replaced by $\oa(d-\epsilon)$:
$$\I_{\oa(d-\epsilon)}(F_{ZZ}F_W^2)=\I_{\oa(d-\epsilon)}(\Delta_F)=\I_{\oa(d-\epsilon)}(\Delta_F+\ga''F_Z^3).$$

It follows that
\be\label{W}L_\Delta(\ga,d-\epsilon)=L_\Delta(\ga,d)-(M-1)\epsilon,  \quad L_{gr}(\ga,d-\epsilon)=L_{gr}(\ga,d)-M\epsilon,\ee and then $$L_\ga(d-\epsilon)-L_\ga(d)=(2M+1)\epsilon >0, \quad L_\ga(d-\epsilon)>L_\ga(d).$$

To prove ($3$), take any $e$, $\tan\theta_{top}<e<d$, $\oa\!:=\oa(e)$. Let us write
\be\label{NM'}N\!:=\deg\I_\oa(\Delta_F+\ga''F_Z^3)(Z,1)\quad (\text{degree of a polynomial in $Z$}).\ee

Let $\epsilon>0$ be sufficiently small, then, like (\ref{W}), we have
$$L_\Delta(\ga,e-\epsilon)=L_\Delta(\ga,e)- N \epsilon,\quad L_{gr}(\ga,e-\epsilon)=L_{gr}(\ga,e)-M \epsilon.$$

\textit{Hence it suffices to show that} $N\leq 3M$. 
Let us write $p(Z)$, $p_1(Z)$, $p_2(Z)$ for 
$$\I_\oa(F_{ZZ}F_W^2)(Z,1), \quad \I_\oa(F_Z^2F_{WW})(Z,1),\quad \I_\oa(F_ZF_WF_{ZW})(Z,1)$$ respectively.
The first equality in $(\ref{3i})$ implies that
$$\I_\oa(\Delta_F)(Z,1)=p(Z)+c_1p_1(Z)+c_2p_2(Z), \quad c_1,c_2\in \C\;\,(\text{possibly zero}).$$

The second equality in $(\ref{3i})$ implies that
$$\I_\oa(\Delta_F+\ga''F_Z^3)(Z,1)=c_3\I_\oa(\Delta_F)(Z,1)+c_4\I_\oa(F_Z^3)(Z,1), \quad (c_3,c_4)\ne (0,0).$$
The right-hand-side has degree $\leq 3M$ since $\I_\oa(F_W)=ahW^{h-1}$. Hence $N\leq 3M$.

To prove ($2$), let us take $e$, $\tan\theta_1<e<\tan\theta_{top}$, and write (compare (\ref{write})) 
$$\I_\oa(F)(Z,W)\! =aZ^{M+1}W^l+a'Z^MW^{l+e}+\cdots, \quad a\ne 0, \quad \oa=\oa(e).$$

Since $e>\tan\theta_1$, ($M+1,l$) cannot be the vertex ($m,0$). Hence $l>0$, and
\be\label{FW}\I_\oa(F_W)(Z,W)=laZ^{M+1}W^{l-1}+\cdots,\ee 
where, since $e<\tan\theta_{top}$, we must have $M+1>0$, and then
\be\label{FZ}\I_\oa(F_Z)(Z,W)=a(M+1)Z^MW^{l}+\cdots.\ee

Since $e>1$, an immediate consequence of (\ref{FW}) and (\ref{FZ}) is
$$ O_W(F_Z(uW^e, W))<O_W(F_W(uW^e,W)),$$ and hence
$$ O_W(\frac{|\ga''F_Z^3|^2}{(|F_Z|^2+|F_W|^2)^3})=2O_W(\ga''), \; \text{a constant}.$$

Now, since $\ga$ is a polar, we actually have $M+1\geq 2$. Hence
$$\I_\oa(F_{ZZ})(Z,W)=aM(M+1)Z^{M-1}W^{l}+\cdots.$$

Similarly, we can write down the formulas for $\I_\oa(F_{ZW})$ and $\I_\oa(F_{WW})$.

\m 

An easy calculation of the determinant $\Delta_F$ gives
$$\I_\oa(\Delta_F)(Z,W)=la^3(M+1)(M+l+1)Z^{3M+1}W^{3l-2}+\cdots,$$ and then we have
$$ \deg \I_\oa(\Delta_F)(Z,1)=3M+1>3M=\deg \I_\oa(\ga''F_Z^3)(Z,1).$$

It follows from this inequality that
$$\I_\oa(\Delta_F+\ga''F_Z^3)=\begin{cases}\I_\oa(\Delta_F)&\quad \text{if $O_\oa(\Delta_F)<O_\oa(\ga''F_Z^3)$},\\ \I_\oa(\ga''F_Z^3)&\quad \text{if $O_\oa(\Delta_F)>O_\oa(\ga''F_Z^3)$},
\\ \I_\oa(\Delta_F)+\I_\oa(\ga''F_Z^3)&\quad \text{if $O_\oa(\Delta_F)=O_\oa(\ga''F_Z^3)$},\end{cases}$$
and hence in every case we have
\be\label{S} \deg \I_\oa(\Delta_F+\ga''F_Z^3)(Z,1)\geq \deg \I_\oa(\ga''F_Z^3)(Z,1).\ee

Now, using (\ref{S}) and the same argument as in the proof of (4), we find
$$L_\ga(e)\geq L_\ga(e-\epsilon), \quad \epsilon>0\;\,\text{sufficiently small}.$$

It remains to prove ($1$). This case involves only the vertex $(m,0)$:
$$O_\oa(F)=me, \quad O_\oa(F_Z)=(m-1)e, \quad O_\oa(F_W)>me-1.$$ 
It follows that
$$O_\oa(\Delta_F)-3O_\oa(F_Z)>e-2>-1, \quad O_\oa(\ga''F_Z^3)-3O_\oa(F_Z)=O(\ga'')>-1.$$

This completes the proof of Lemma \ref{32}.\end{proof}

\section{Proofs of Theorems A and B}\label{proofs}

\textit{Let us prove Lemma} \ref{LemA}. First, consider the case $d\!:=d_{gr}(\ga)>1$.  We already know $L_\ga=-d$ in Lemma \ref{32}.
Moreover, $R_\ga(u)$ is clearly defined for all $u$ since
$$D_{(\ga,d)}(u)\geq |ah|^2>0, \quad a,\,h \;\,\text{as in}\;\, (\ref{3k}).$$

Let $N$, $M$ be as in (\ref{write}) and (\ref{NM'}), where $e=d$. Then, by (\ref{3ee}),
$$N=\deg \I_{\oa(d)}(F_{ZZ}F_W^2)(u,1)=M-1<3M=3\deg\I_{\oa(d)}(F_Z)(u,1).$$
Hence $\lim_{u\ra \infty}R_\ga(u)=0$.

\m

Now suppose $d=1$. For simplicity, we write $H\!:=H_m(z,w)$. By Euler's Theorem,
\be\label{Euler}\Delta_H(z,w)=-\frac{m}{m-1}\cdot H\cdot Hess(H), \quad Hess(H)\!:=\begin{vmatrix}H_{zz}&H_{zw}\\H_{wz}&H_{ww}\end{vmatrix}.\ee

Note that $H_m(z,w)$ has at least two different factors, for otherwise we would have
$$H_m(z,w)=c(z-u_0w)^m, \quad d_{gr}(\ga)>1,$$a contradiction. It is easy to see that, as a consequence,  $$Hess(H)\not\equiv 0, \quad \Delta_H\not \equiv 0.$$

\m

Take any $u_0$. If $H(u_0,1)\ne 0$, then $\Grad\!H(u_0,1)\ne 0$, $R_\ga(u_0)$ is clearly defined.

\m

If $z-u_0w$ is a factor of $H$ of order $k$, $k\geq 1$, then it is one of $\Delta_H$ of order $3k-2$,
$$\Delta_H(u,1)=(u-u_0)^{3k-2}Q_1(u), \quad |\Grad\!H(u,1)|=|u-u_0|^{k-1}Q_2(u),$$where $Q_1(u_0)Q_2(u_0)\ne 0$, $u$ near $u_0$. Hence $|u-u_0|^2$ divides $R_\ga(u)$, $R_\ga(u_0)=0$.

\m

Thus $R_\ga(u)$ is defined for all $u$. It remains to show $L_\ga=-1$ and $\lim_{u\ra \infty}R_\ga(u)=0$.

\m

We can assume $\ga\in \F_{1^+}$. In terms of $(Z,W)$ we have
$$O_\oa(\Delta_H)=3m-4<O_\oa(\ga''F_Z^3)=3(m-1)+O(\ga)-2.$$Hence $$O_\oa(\Delta_H+\ga''F_Z^3)=3m-4, \quad L_\ga=(3m-4)-3(m-1)=-1.$$

We also have $\lim_{u\ra \infty}R_\ga(u)=0$, since
$$\deg \Delta_H(u,1)\leq 3m-4<3(m-1)=3\deg H_Z(u,1).$$

\m

\textit{Now we  first prove Theorem B, after which we prove Theorem A.}

\m

Let $\GC(\ga_*)$ be given, $1<d_{gr}(\ga_*)<\infty$. Take $\ga_j$ in (\ref{ff}), and suppose
\be\label{GCS} \GC(\ga_{j*})\subseteq \GC(\ga_*), \quad 1<d_{gr}(\ga_{j*})\leq \infty.
\ee

Amongst the conjugates of $\ga_j$, there is one, say $\ga_k$, $\ga_{j*}=\ga_{k*}$, such that 
$$O_y(\ga-\ga_k)\geq d_{gr}(\ga_*).
$$ 

It follows from Lemma \ref{LH} (i) that $$O_y(f(\ga_k(y),y))=O_y(f(\ga(y),y)),\quad d_{gr}(\ga_*)= d_{gr}(\ga_{j*})= d_{gr}(\ga_{k*}).$$
The inclusion in (\ref{GCS}) is an equality, $d_{gr}(\ga_{j*})=\infty$ cannot happen, $\GC(\ga_*)$ is minimal.

\m

Recall that two infinitesimal lines are either disjoint, or one contains the other (see (\ref{TWO})). The first statement of Theorem B is proved.

\m 

For the remaining statements,
let $r$ be as in (\ref{HG}). By Theorem 2.1 in \cite{kuo-par} (Lemma $3.3$ in \cite{kuo-lu}), there are exactly $r-1$ polars $\ga_j$ (counting multiplicities), for each of which
\bee O(\ga_j-\zeta_i)=1, \quad  1\leq i \leq m,\;\, \text{$\zeta_i$ as in (\ref{ff})}.
\eee

Thus, if $r>1$, then, by (\ref{tan}), $d_{gr}(\ga_j)=1$, $\GC(\ga_{j*})=\C_{enriched}$ is a gradient canyon.

If $r=1$, then, by (\ref{tan}), $d_{gr}(\ga)>1$ for all polar $\ga$, $\C_{enriched}$ is not a gradient canyon.

\m

Finally, suppose $1<r\leq m$. 

For the roots $\zeta_i$ of $f$, $J^{(1)}(\zeta_i)$, $1\leq i \leq m$, are the roots of $H_m(z,w)$. Hence if $f$ has more than $r$ distinct roots, then there exist two roots $\zeta_i$, $\zeta_k$ such that
\bee
 1<O(\zeta_i(y)-\zeta_k(y))<\infty.
\eee

Then, by the same theorem as above, there is a polar $\ga_j$ such that 
\bee 1<O(\zeta_i-\zeta_k)=O(\zeta_i-\ga_j)=O(\zeta_k-\ga_j)\geq O(\ga_j-\zeta_l), \;\, 	1\leq l\leq m.
\eee
It follows that $1< d_{gr}(\ga_{j*})<\infty$, $\GC(\ga_{j*})$ is minimal,  $\C_{enriched}$ is not.

On the other hand, if $f$ has exactly $r$ roots, every polar has $d_{gr}=1$ or $\infty$.
  
\m

\textit{Now we prove Theorem A.} We use Lemma \ref{32} and the following lemma.

\begin{lem}Consider $R_{(\al,e)}(u)$ in (\ref{LAP}). Take $c\in \C$ such that $R_{(\al,e)}(c)\ne 0$. Then
\bee
\beta(y)=[J^{(e)}(\al)(y)+cy^e]+\cdots\implies K_*(\beta_*)=K_*((\al+cy^e)_*).
\eee
\end{lem}
\begin{proof}According to (\ref{LAP}), we can write
\begin{eqnarray}\label{replace}&
R_{(\al,e)}(u)=|\frac{P(u)}{Q(u)}|, \quad P(u),\; Q(u)\in \C[u],&
\end{eqnarray}
where $P(c)\ne 0$, $Q(c)\ne 0$. We then consider the Taylor expansion at $c$,
\begin{eqnarray}&
\frac{P(u)}{Q(u)}=a_0+a_1(u-c)+\cdots, \quad |a_0|=R_{(\al,e)}(c).&
\end{eqnarray}

We can write $\beta(y)$ in the form
\bee
\beta(y)=\al(y)+[c+\nu(y)]y^e, \quad \nu(y)\in \F_0, \; \nu(0)=0,
\eee

Let us replace $u$ by $\nu(y)+c$ in (\ref{replace}), then
\bee
R_{(\al,e)}(\nu(y)+c)=R_{(\al,e)}(c)+\cdots.
\eee
It follows that
\bee
K(\beta(y), y)=2R_{(\al,e)}(c)y^{L_\al(e)}+\cdots, \quad K_*(\beta_*)=(2R_{(\al, e)}(c), L_\al(e)).
\eee 

This completes the proof of the Lemma.
\end{proof}

Now consider Definition \ref{DAB}. Condition ($1$) follows from the above.
The function $R_{\ga}(u)$, which is not constant, attains  a local maximum at $c$, hence ($2$) holds. Condition ($3$) follows from Lemma \ref{32}($4$). Hence $\mathcal{D}^{(d)}(\beta_*;0)$  is an A'Campo bump.

To show all A'Campo bumps can be obtained in this way, let an A'Campo bump $\mathcal{D}^{(e)}(\beta_*;0)$ be  given. Take a polar, $\ga$, such that $O_y(\beta-\ga)$ is maximum. Using Lemma\,\ref{32}, we see that unless $\beta_*\in \GC(\ga_*)$, there would be a contradiction. The proof of Theorem A is complete.

\section{Proof of Theorems C and D}\label{CD} 

\textit{We prove Theorem C}. The case $d_{gr}(\ga_*)=\infty$ is trivial. We assume $1<d_{gr}(\ga_*)<\infty$.

\m

Consider $\PP(F)$, $F(Z,W)$ as in (\ref{3b}).
In the first place, along $F(Z,W)=c$,
\be\label{ZW} F_ZdZ+F_WdW=0, \quad dS=\frac{|F_Z|^2+|F_W|^2}{|F_W|^2}\,\frac{dZ\wedge d\bar{Z}}{-2\sqrt{-1}}.\ee

Let us set $Z=uW^d$, $d\!:=d_{gr}(\ga)>1$, then
$$ dZ=W^ddu+(ud) W^{d-1}dW.$$ 
Eliminating $dW$,
\be\label{ZWW} dZ=\frac{(WF_W)}{WF_W+(ud)W^dF_Z}\cdot W^d\cdot du.\ee

Let $h$, $a$ be as in (\ref{3k}). See Fig.\,\ref{fig:npZ}. Then, for generic $u$,
\be\label{FFWW} O_W(F_Z(uW^d,W))=h-1,
\ee
and for all $u$,
\bee WF_W(uW^d,W)=haW^h[1+o(W))].
\eee

\begin{convention}Here, as in Calculus, $o(W)$ represents a function $\varphi(W)$, $\varphi(0)=0$.\end{convention}

It follows from the above that for all $u$,
$$O_W(WF_W)=h<h-1+d\leq O_W(W^dF_Z).
$$ 
Hence (\ref{ZWW}) can be rewritten as
\be\label{CBR}dZ=W^d[1+o(W)]du.\ee

Now take the weight system $\oa(d)$: $\oa(Z)=d$, $\oa(W)=1$. Consider the polynomial
\bee p(u)\!:=\I_{\oa(d)}(F_Z)(u,1). \quad \text{(As defined in \textit{Notation}\,\ref{no}.)}
\eee

Note that, in fact
\bee
p(0)=0, \quad \deg p(u)=m(\GC(\ga_*)),
\eee
the latter was defined in (\ref{multi}). By (\ref{LD}), (\ref{FFWW}), (\ref{3e}),
$$D_{(\ga,d)}(u)=|p(u)|^2+|ha|^2, \quad N_{(\ga,d)}(u)=|ha|^4|p^\prime (u)|^2.$$ 

Then, by (\ref{KNF}), (\ref{3e}), (\ref{ZW}) and (\ref{CBR}),
\be\label{KK} \begin{split}K dS&=\frac{2|F_{ZZ}F_W^2|^2+\cdots}{(|F_Z|^2+|F_W|^2)^2}\cdot\frac{1}{|F_W|^2}\cdot\frac{dZ\wedge d\bar{Z}}{-2\sqrt{-1}}\\&=\{\frac{2|p^\prime(u)|^2|ha|^2}{(|p(u)|^2+|ha|^2)^2}+o(W)\}\frac{du\wedge d\bar{u}}{-2\sqrt{-1}}.\end{split}\ee

To compute the integrals in (\ref{measure}) using (\ref{KK}), we must know the number of sheets of $\SU_c\cap H_{orn}^{(d)}(\ga_*;\rho;\eta),$ lying over the $u$-plane, $i.e.$ the number of values of $W$ when $u$ is given.

\m

For this, in our situation let us first consider the special case where all $\ga_j(y)$ are \textit{integral} power series.
In this case there is no conjugation,
$$h\!:=O_y(f(\ga(y), y))=O_W(F(0,W)),\;\,\text{an integer}.$$

When $u$ is given, there are $h$ distinct values of $W$ found by solving the equation
\bee c=aW^h +\cdots,\quad (\text{$c$ is the value of the level: $F=c$})
\eee
where only those terms of order $>h$  depend on $u$. 
\textit{Hence the number of sheets is $h$.} 

\m

Let $R>0$ be given, and fixed. Then
\be\label{hhh}
\lim_{W\ra 0}\int_{H_{orn}^{(d)}(\ga_*;R;\eta)\cap
\mathcal{S}_c}K dS=h\cdot\int _{|u|\leq R}\frac{2b^2|p^\prime(u)|^2}{[b^2+|p(u)|^2]^2}\,\frac{du\wedge d\bar{u}}{-2\sqrt{-1}},
\ee
where $b\!:=|ha|.$ To compute the right-hand side, we write the polynomial $p(u)$ as 
$$p(u)\!:=U(x,y)+\sqrt{-1}\,V(x,y), \quad u=x+\sqrt{-1}\,y\in \C,
$$where $U$, $V$ satisfy the Cauchy-Riemann equations. Using the latter we have
$$\int\frac{2b^2|p^\prime(u)|^2}{(b^2+|p(u)|^2)^2}\,\frac{du\wedge d\bar{u}}{-2\sqrt{-1}}=
\int \frac{2b^2dU\wedge dV}{(b^2+U^2+V^2)^2}\quad (\text{indefinite integrals}).$$

Now $u\mt p(u)$ is a $\deg p(u)$-sheet branch covering map of $\C$. Letting $R\ra \infty$,
\be\label{hhhh}\begin{split}
\int_\C\frac{2b^2|p^\prime(u)|^2}{(b^2+|p(u)|^2)^2}\,\frac{du\wedge d\bar{u}}{-2\sqrt{-1}}&=
\deg p(u)
\cdot 
\int_{\R^2} \frac{2b^2 dU\wedge dV}{(b^2+U^2+V^2)^2}\\
&=2\pi \cdot\deg p(u)=2\pi \cdot m(\GC(\ga_*)).\end{split}
\ee

On the other hand, by Lemma \ref{LH}($i$), the following implication is true: for any $\al\in \F$,
\be\label{ing}
O_y(\al-\ga)\geq d_{gr}(\ga_*)\implies O_y(f(\al(y),y))=h.
\ee

 As a consequence of this we have ($\eta$ small enough):
 \bee
 F(Z,W)=0, \; (Z,W)\in H_{orn}^{(d)}(\ga_*;R;\eta)\implies (Z,W)=0.
\eee 

It follows that for $(Z,W)\in H_{orn}^{(d)}(\ga_*;R;\eta)\cap\mathcal{S}_c$ the following conditions are equivalent:
\bee
(i)\,(Z,W)\ra 0, \quad  (ii)\,c\ra 0, \quad (iii)\, W\ra 0.
\eee

Hence (for \textit{fixed} $R$)
\bee
\lim_{W\ra 0}\int_{H_{orn}^{(d)}(\ga_*;R;\eta)\cap
\mathcal{S}_c}K dS=
\lim_{c\ra 0}\int_{H_{orn}^{(d)}(\ga_*;R;\eta)\cap\mathcal{S}_c}K dS,
\eee
and this value is \textit{independent} of $\eta$. Therefore
\bee
\lim_{c\ra 0}\int_{H_{orn}^{(d)}(\ga_*;R;\eta)\cap\mathcal{S}_c}K dS =\lim_{\eta\ra 0}\lim_{c\ra 0}\int_{H_{orn}^{(d)}(\ga_*;R;\eta)\cap\mathcal{S}_c}K dS=\mer_f(\mathcal{D}^{(d)}(\ga_*;R)).
\eee
(The last equality is the definition (\ref{measure}).) 

Thus, by (\ref{hhh}), (\ref{hhhh}), we have
\begin{eqnarray*}&
\mer_f(\GC(\ga_*))=\lim_{R\ra \infty}\mer_f(\mathcal{D}^{(d)}(\ga_*;R))=h\cdot 2\pi\cdot m(\GC(\ga_*)).&
\end{eqnarray*}

Now, take any polar $\ga_j$ such that $\ga_{j*} \in \GC(\ga_*)$. Then, by (\ref{ing}),
\be\label{ano} O_y(f(\ga_j(y),y))=h.
\ee  
 The number of such $\ga_j$ is (by definition) $m(\GC(\ga_*))$. Hence, by (\ref{M}), (\ref{ano}),
$$\mu_f(\GC(\ga_*))= m(\GC(\ga_*))\cdot (h-1).$$ 
It is then trivial to derive (\ref{ttl}).
 
The general case, where $\ga$ is fractional, can be derived from the above by the substitution $y\ra y^{M}$, where $M$ is divisible by all $m_{puiseux}(\ga_j)$, so that all $\ga_j(y^{M})$ are integral. 

In this way $h$ is magnified to $Mh$, and $\SU_c$ is blown up to $M$ copies of itself. Let $R\ra\infty$,
\begin{eqnarray*}&\mer_f(\GC(\ga_*))=\frac{1}{M}\cdot 2\pi (Mh)\cdot m(\GC(\ga_*))=2\pi[\mu_f(\GC(\ga_*)+m(\GC(\ga_*))].&
\end{eqnarray*}
This completes the proof of Theorem C. 

\m
\newpage
\textit{Next we prove Theorem D}.

Let a disc $D_{isc}(0;\eta)$ be given (Definition\,\ref{disc}). Take constants  $\rho$, $R$, $\rho>0$ sufficiently small, $R$ sufficiently large. We use $C_\emptyset$, $C_\emptyset '$
to denote \textit{suitable non-zero constants}; and
\bee
A\approx B\quad  \text{means}\quad
0<C_\emptyset \leq A/B \leq C'_\emptyset.
\eee

Let $0\leq a <b\leq \infty$ be given. In this section we shall also use 
the short-hand
\be\label{Hdef}
H_{orn}(a,b)\!:=\{(Z,W)\in D_{isc}(0;\eta)\,|\,R|W|^b\leq |Z|\leq\rho|W|^a\},
\ee
with the convention that $R|W|^\infty=0$. If we set $Z=uW^a$, then
\bee
(Z,W)\in H_{orn}(a,b) \Longleftrightarrow R|W|^{b-a}\leq |u|\leq \rho.
\eee

Now, let us  consider the Newton Polygon $\PP(G)$ of a general $G(Z,W)$, with edges $\hat{E}_i$,  and angles $\hat{\theta_i}$, $0\leq i \leq v$, where $\hat{E}_0$ is horizontal, $\hat{E}_v$ is vertical, $\hat{\theta}_v=\pi/2$. We write 
\bee \hat{e}_i:= \tan\hat{\theta}_i, \quad i=0,...,v.
\eee

For example, if $G=F$ in Fig\,\ref{fig:npF}, then $\hat{E}_{v-1}=E_{top}$, $\hat{E}_v$ is the vertical half-line at $(0,h)$. 

\begin{vlem}\label{SDI}Let $(k,q)$ be a given vertex of $\PP(G)$, say the left vertex of $\hat{E}_i$ and the right vertex of $\hat{E}_{i+1}$, where $0\leq i \leq v-1$. Then
\be\label{Rrho} |G(Z,W)|\approx |Z|^k|W|^q, \quad (Z,W)\in H_{orn}(\hat{e}_i,\hat{e}_{i+1}).
\ee

In particular, for any given $\hat{e}$, $\hat{e}_i\leq \hat{e} <\hat{e}_{i+1}$, we have
\bee
 |G(Z,W)|\approx |Z|^k|W|^q, \quad (Z,W)\in H_{orn}(\hat{e},\hat{e}_{i+1}).
\eee
\end{vlem}

\begin{proof}Note that (\ref{Rrho}) is clearly true if $G$ is of the special form
\be\label{FO}
G(Z,W)=unit\cdot Z^kW^q \quad \text{($|Z|$, $|W|$ sufficiently small)}.
\ee

For the general case, we can find a linear isomorphism,
taking  
$\hat{E}_{i+1}$ to a vertical half-line and $\hat{E}_{i}$ to a horizontal one. 
Then $G(Z,W)$ is transformed to a function of the special form (\ref{FO}).
\end{proof}

\begin{cor}\label{vcor}If a dot $(P,Q)$ lies on or above $\PP(G)$, then $G$ ``dominates" $Z^PW^Q$  in every horn domain. That is to say, for each $i$, $0\leq i\leq v-1$, 
\be\label{follow}|G(Z,W)|^2+|Z^PW^Q|^2\approx |G(Z,W)|^2, \quad (Z,W)\in H_{orn}(\hat{e}_i,\hat{e}_{i+1}).
\ee
\end{cor}
(This does not mean $``\approx"$ is true \textit{for all} $(Z,W)$. Because of $\rho$ and $R$, a number of ``horn strips" are excluded.)

\begin{Elem}Let $\oa\!:=\oa(e)$ be given. Consider the weighted Taylor expansion
\bee
G(Z,W)=H^{(\oa)}_{D_1}(Z,W)+\cdots +H^{(\oa)}_{D_k}(Z,W)+\cdots,\quad D_k=\deg_{\oa}H^{(\oa)}_{D_k},
\eee
where $H^{(\oa)}_D$ denotes a weighted homogeneous form of weighted degree $D$.

Then the weighted Taylor expansion of $WG_W+eZG_Z$ is
\bee
WG_W+eZG_Z= D_1\cdot H^{(\oa)}_{D_1}(Z,W)+\cdots +D_k\cdot H^{(\oa)}_{D_k}(Z,W)+\cdots.
\eee 
In particular, $G$ and $WG_W+eZG_Z$ have the same set of Newton dots,
\be\label{PPPP}
\PP(G)=\PP(WG_W+eZG_Z).
\ee
\end{Elem}This lemma is an immediate consequence of Euler's Theorem (weighted version).

\m

In the following corollary we take $G$ to be $F$, $1<d_{gr}(\ga_*)<\infty$. The edges and angles of $\PP(F)$ are $E_i$ and $\theta_i$, respectively, $0\leq i \leq v$. We write
\be\label{ei}
e_i\!:=\tan\theta_i, \quad 0\leq i \leq v.\quad  (\tan\theta_{v}=\infty,\; E_{v-1}=E_{top}.)
\ee

\begin{cor}Take any $e$, $e_i\leq e<e_{i+1}, 0\leq i\leq v-1.$ Then
\be\label{6.5}
\frac{|WF_W(uW^{e},W)|}{|WF_W(uW^{e},W)+euW^{e}F_Z(uW^{e},W)|}\leq C_\emptyset, \quad R|W|^{e_{i+1}-e}\leq |u|\leq \rho.
\ee
\end{cor}

\begin{proof}Let  $(k,q)$ be the common vertex of $E_i$, $E_{i+1}.$  Then, by the above two lemmas,
\bee |F(uW^{e}, W)|\approx |WF_W+euW^{e}F_Z|\approx |u|^k|W|^{q+ke}.
\eee

Every dot of $WF_W$ lies on or above $\PP(F)$. Hence
\bee
|WF_W(uW^{e}, W)|\leq C_\emptyset' |u|^k|W|^{q+ke},
\eee
by Corollary\,\ref{vcor}. This completes the proof.  
\end{proof}

\textit{Next consider} $F_Z(Z,W)$, which has no Newton dot of the form $(0,q)$. We consider
\be\label{phi}
G(Z,W)\!:=F_Z(Z,W)+W^{h-1},\quad  \text{$h$ as in (\ref{3k})}.
\ee

This time we denote the edges of $\PP(G)$ by $E_i'$, angles by $\theta_i'$, $0\leq i \leq p+1$, and
\be\label{eip}
e_i'\!:=\tan\theta_i', \quad 0\leq i \leq p+1.
\ee
In particular,
\bee
\tan\theta_p'=d_{gr}(\ga_*),\quad \tan\theta_{p+1}'=\infty.
\eee

\begin{lem}
The numbers $\{e_i\}$ in (\ref{ei}) and $\{e_i'\}$ in (\ref{eip}) are related as follows: 
\bee p\geq v-1;  \quad e_{v-1}'\geq e_{v-1};\quad e_i'=e_i , \;0\leq i \leq v-2\;\text{(if $v\geq 2$)};\eee
where $e_{v-1}'>e_{v-1}$ if and only if $(\widehat{m}_{top},\widehat{q}_{top})=(m_{top}, q_{top})$ (see Fig.\,\ref{fig:npZ}).
\end{lem}

This relationship can be seen clearly in Fig.\,\ref{fig:npZ}. Note that $(k,q)$, $k>0$, is a Newton dot of $F$ if, and only if, it is one of $ZF_Z$.

\begin{cor}\label{cor62} Given $i$, $0\leq i\leq p-1$, $F_Z$
``dominates'' $F_W$ in $H_{orn}(e_i',e_{i+1}')$. That is,
\be\label{domin}
|F_Z|^2+|F_W|^2\approx |F_Z|^2, \quad (Z,W)\in H_{orn}(e_i',e_{i+1}').
\ee

However, $F_Z$ ``is dominated by'' $F_W$ in $H_{orn}(e_p',e_{p+1}')$. That is,
\be\label{sdomi}
|F_Z|^2+|F_W|^2\approx |F_W|^2, \quad (Z,W)\in H_{orn}(e_p',e_{p+1}').
\ee
\end{cor}

\begin{proof}
All dots of $F_W$ lie on or above $\PP(G)$,  $G$ defined in (\ref{phi}). On the other hand,
\bee
|F_Z(Z,W)|\approx |G(Z,W)|\quad (Z,W)\in H_{orn}(e_i',e_{i+1}'),
\eee
where $0\leq i \leq p-1$. Hence (\ref{domin}) follows from Corollary
\,\ref{vcor}.

Because $(0,h-1)$ is a vertex of $\PP(G)$, (\ref{sdomi}) is a consequence of  the Vertex Lemma and Corollary\,\ref{vcor}.
\end{proof}

We now establish four lemmas, Lemma\,\ref{Lemma1} to Lemma\,\ref{Lemma5}, then derive Theorem D.

\begin{lem}\label{Lemma1}Let $\ga_*$ be a given polar, $1<d_{gr}(\ga_*)<\infty$. Take $e'$, $1<e'<d_{gr}(\ga_*)$.

Take $i$ such that $e_{i-1}'\leq e'<e_i'$ (hence $i\leq p$). Let $\epsilon>0$ be sufficiently small. Then
\be\label{kkk}
\mer_f(\mathcal{D}^{(e')}(\ga_*,\epsilon))=\mer_f(\mathcal{L}^{(e_i')}(\ga_*)),
\ee 
where $\mathcal{D}$, $\mathcal{L}$ are defined in (\ref{ID}), (\ref{IL}). In particular,
\be\label{IPA}
e_{p-1}'\leq e' <d_{gr}(\ga_*)\;\,(\text{and}\;e'>1)\implies \mer_f(\mathcal{D}^{(e')}(\ga_*,\epsilon))=\mer_f(\GC(\ga_*)).
\ee
\end{lem}

\begin{lem}\label{Lemma2}  Let $\mathcal{D}^{(e)}(\al_*,\epsilon)$ be a given infinitesimal disc, $1<e<\infty$, $\epsilon>0$.

Suppose this disc is disjoint from every minimal gradient canyon (see (\ref{perm})). Then
\be\label{zeroo}
\mer_f(\mathcal{D}^{(e)}(\al_*,\epsilon))=0.
\ee
(The disc may contain singleton canyons.)

On the other hand, if there is a polar $\ga_*$ such that
\bee
\mathcal{D}^{(e)}(\al_*,\epsilon)\subset \GC(\ga_*), \quad e>d_{gr}(\ga_*)> 1,
\eee
then \be\label{zerooo}
\mer_f(\mathcal{L}^{(e)}(\al_*))=0.\ee
\end{lem}

\textit{We now prove Lemma \ref{Lemma1}}.

\m

We show the difference of the two sides of (\ref{kkk}) is zero. With the substitution $Z=uW^{e'}$, this amounts to showing that
\be\label{KDS}
\lim_{R\ra \infty}\lim _{W\ra 0}\int_{R|W|^{\delta_i}\leq |u|\leq \epsilon}KdS=0, \; \,\delta_i\!:=e_i'-e', \,\text{$\epsilon>0$ a small constant}.
\ee

Now, like (\ref{ZWW}), we have
 \bee
 dZ=QW^{e'}\,du, \quad Q\!:=\frac{WF_W}{WF_W+e'uF_Z},
 \eee
 where $|Q|\leq C_\emptyset$ by (\ref{6.5}). Hence
 \be\label{KKSS}
 KdS=\frac{2|\Delta_F+\ga''F_Z^3|^2}{(|F_Z|^2+|F_W|^2)^2}\cdot \frac{|Q|^2}{|F_W|^2}\cdot |W|^{2e'}\cdot \frac{du\wedge d\bar{u}}{-2\sqrt{-1}}.
 \ee

We divide the proof into two cases:
\be\label{III}
(\text{I})\; \tan\theta_1'=\tan\theta_1, \quad (\text{II})\; \tan\theta_1'>\tan\theta_1.
\ee

Let us first consider (I). This is the only difficult case.

\m

In the first place, since $d_{gr}(\ga_*)>1$, we know $p\geq 2$. (See Fig.\,\ref{fig:npZ}.)

\m

Consider the weight system $\oa(e')$ (where $e_{i-1}'\leq e'<e_i'$). Let us write
\be\label{writefz}
\I_{\oa(e')}(F_Z)(Z,W)\!:=c_kZ^kW^q+c_{k+1}Z^{k+1}W^{q-e'}+\cdots,\quad c_k\ne 0,
\ee
where $k\geq 1,$ since $\ga$ is a polar. If $e'\geq e_1'$, then $q\geq 1$. 

Note that $(k,q)$ is the common vertex of $E_{i-1}'$ and $E_i'$, and
\bee 
e_{i-1}'<e'<e_i'\implies 
\I_{\oa(e')}(F_Z)(Z,W)=c_kZ^kW^q, \; \text{a monomial}.
\eee

 To compute (\ref{KDS}), we use the expansion (\ref{Del}). We define $K_i$, $1\leq i \leq 4$, and show
 \bee
 \lim_{R\ra \infty}\lim_{W\ra 0}\int_{R|W|^{\delta_i}\leq |u|\leq \epsilon} K_i\cdot\frac{du\wedge d\bar{u}}{-2\sqrt{-1}}=0, \quad 1\leq i \leq 4.
 \eee
 It will then be obvious that (\ref{KDS}) is true (an easy use of the triangle inequality). 

To begin, we define
\bee
K_1\!:=\frac{|F_{ZZ}F_W^2|^2}{(|F_Z|^2+|F_W|^2)^2}\cdot \frac{|Q|^2}{|F_W|^2}\cdot |W|^{2e'}.
\eee

Using the fact that $F_Z$ dominates $F_W$ (Corollary\,\ref{cor62}), we have
\bee
K_1\leq C_\emptyset \frac{|ZF_{ZZ}|^2}{|F_Z|^2} \frac{|F_W|^2}{|F_Z|^2}\cdot \frac{|W|^{2e'}}{|Z|^2}, \quad R|W|^{\delta_i}\leq |u|\leq \epsilon.
\eee
Then, since $ZF_{ZZ}$ and $F_Z$ have the same Newton dots,
\be\label{K}
K_1\leq C_\emptyset'\frac{|F_W|^2}{|F_Z|^2}\cdot \frac{|W|^{2e'}}{|Z|^2}, \quad R|W|^{\delta_i}\leq |u|\leq \epsilon.
\ee

Now let us first consider the case $e_{p-1}'\leq e'<d_{gr}(\ga_*)$.

Applying the Vertex Lemma to the vertex $(k,q)$,  we have
\bee
|F_Z|\approx |u|^k|W|^{q+ke'}, \quad R|W|^{\delta_p}\leq |u|\leq \epsilon.
\eee

Since $(0,h-1)$ is the only dot of $F_W$ on $L^*$, we also have
\bee
 |F_W|\approx |W|^{h-1},\quad h-1=q+kd_{gr}(\ga_*). \quad (\text{See Fig,\,\ref{fig:npZ}}.)
 \eee
 
It follows that
\bee
K_1\leq C_\emptyset'\frac{|F_W|^2}{|F_Z|^2}\cdot \frac{|W|^{2e'}}{|u|^2|W|^{2e'}}\leq C_\emptyset ''\frac{|W|^{2k\delta_p}}{|u|^{2(k+1)}}, \quad R|W|^{\delta_p}\leq |u|\leq \epsilon.
\eee

We use polar coordinates in the following computation:
\be\label{POLAR}\int_{R|W|^{\delta_p}\leq |u|\leq \epsilon}K_1\frac{du\wedge d\bar{u}}{-2\sqrt{-1}}
\leq C_\emptyset\int_{R|W|^{\delta_p}}^\epsilon\frac{ |W|^{2k\delta_p }\cdot r dr}{r^{2(k+1)}}= C_\emptyset[\frac{1}{R^{2 k}}-\frac{|W|^{2k\delta_p}}{\epsilon^{2k}}],
\ee
where $u=r\exp{\sqrt{-1}\theta}$. Thus (with $\epsilon>0$ fixed)
\bee
\lim_{R\ra \infty}\lim_{W\ra 0}\int_{R|W|^{\delta_p}\leq |u|\leq \epsilon}K_1\cdot \frac{du\wedge d\bar{u}}{-2\sqrt{-1}}
=0.
\eee

\m

Next suppose $e_{i-1}'\leq e'<e_i'$, where $2\leq i \leq p-1$. 

As pointed earlier, $k\geq 1$,  $q\geq 1$. Hence
\begin{eqnarray}&
\frac{1}{q}W\frac{\partial}{\partial W}(Z^kW^q)=\frac{1}{k}Z\frac{\partial}{\partial Z}(Z^kW^q)=Z^kW^q.&
\end{eqnarray}
It follows that
\bee |WF_W|\approx |ZF_Z|, \quad  R|W|^{\delta_i}\leq |u|\leq \epsilon.
\eee

Then,  by (\ref{K}),
\bee
K_1\leq C_\emptyset |W|^{2(e'-1)},\quad  \lim_{W\ra 0}\int_{R|W|^{\delta_i}\leq |u|\leq \epsilon}K_1 \frac{du\wedge d\bar{u}}{-2\sqrt{-1}}
=0.
\eee
The equality is \textit{independent} of $R$. Letting $R\ra \infty$, the limit is of course still zero.

\m

Now we define $K_2$, $K_3$,  where $e_{i-1}'\leq e'\leq e_i'$,  $2\leq i\leq p$, 
\bee
K_2\!:=\frac{|F_{ZW}F_ZF_W|^2
|Q|^2
}{(|F_Z|^2+|F_W|^2)^2}
\cdot \frac{|W|^{2e'}}{|F_W|^2}, 
\quad K_3\!:=\frac{|F_{WW}F_Z^2|^2
|Q|^2
}{(|F_Z|^2+|F_W|^2)^2}
\cdot \frac{|W|^{2e'}}{|F_W|^2}.
\eee

Since $F_Z$ dominates $F_W$ in the horn domain $ R|W|^{\delta_i}\leq |u|\leq \epsilon$, we have
\bee
K_2 \leq C_\emptyset \frac{|WF_{ZW}|^2}{|F_Z|^2} |W|^{2(e'-1)},\quad  K_3\leq C_\emptyset \frac{|WF_{WW}|^2}{|F_W|^2} |W|^{2(e'-1)}.
\eee
It follows that 
\bee
K_2 \leq C_\emptyset'|W|^{2(e'-1)},\quad
 K_3\leq  C_\emptyset'|W|^{2(e'-1)}, 
\eee
where $e'-1>0$. Hence
\bee
\lim_{W\ra 0}\int_{R|W|^{\delta_i}\leq |u|\leq \epsilon}K_2
\frac{du\wedge d\bar{u}}{-2\sqrt{-1}}
=0,\quad \lim_{W\ra0}\int_{R|W|^{\delta_i}\leq |u|\leq \epsilon}K_3
\frac{du\wedge d\bar{u}}{-2\sqrt{-1}}
= 0.
\eee

Again, the values are independent of $R$. Letting $R\ra \infty$, the limits are still zero.

\m

Finally, we consider $K_4$, $e_{i-1}'\leq e'\leq e_i'$,  $2\leq i\leq p$, where
\bee K_4\!:=\frac{|\ga''F_Z^3|^2|Q|^2
}{(|F_Z|^2+|F_W|^2)^2}\cdot \frac{|W|^{2e'}}{|F_W|^2}
\leq C_\emptyset \frac{|ZF_Z|^2}{|WF_W|^2}\cdot \frac{|\ga''|^2|W|^{2(e'+1)}}{|Z|^2}.
\eee
We have
\bee
K_4 \leq C_\emptyset' \frac{ |W|^{2[1+O(\ga'')]}}{|u|^2}, \quad \lim_{W\ra 0}\int_{R|W|^{\delta_i}\leq |u|\leq \epsilon} K_4\frac{du\wedge d\bar{u}}{-2\sqrt{-1}}
=0, 
\eee
the last equation follows from an elementary result in Calculus:
\begin{eqnarray*}&
\lim_{W\ra 0}\,|W|^\nu\cdot \ln |W|=0, \quad \nu\!:=1+O(\ga'')>0.&
\end{eqnarray*}

By (\ref{KKSS}), the above results obviously imply (\ref{KDS}).

\m

To complete the proof of (I) it remains to consider the case $1<e'<e_1'=e_1$. 
 
In this case, there is no need to change variable (\ref{3b}), (\ref{writefz}) reduces to
\bee
\I_{\oa(e')}(f)(z,w)=c_{m-1}z^{m-1},
\eee
and then $f_z$ dominates $f_w$, $zf_z$ dominates $wf_w$, etc.. It follows that
\bee
K_j\leq C_\emptyset |w|^{2(e'-1)},\; 1\leq j \leq 3, \quad R|W|^{\delta_1}\leq |u|\leq\epsilon,
\eee
while $K_4$ need not be considered. Since $e'>1$, we have
\bee
\lim_{w\ra 0}\int_{R|w|^{\delta_1}\leq |u|<\epsilon} K_j\frac{du\wedge d\bar{u}}{-2\sqrt{-1}}
=0, \quad 1\leq j \leq 3.
\eee

Again this implies (\ref{KDS}).

\m 

To complete the proof of Lemma\,\ref{Lemma1} it remains to consider the case (II) in (\ref{III}).

By Theorem\,2.1 in \cite{kuo-par}, (II) can  happen only if $f(z,w)$ has weighted Taylor expansion
\bee
f(z,w)=(C_\emptyset z^m+C_\emptyset' w^n)+\cdots, \quad \oa(z)\!:=n/m=e_1 \geq 1, \; \oa(w)=1,
\eee
where $``\cdots"$ means ``higher weighted order terms".

In this case $f_z$ dominates $f_w$. The proof is the same as above.

\m

\textit{Next we prove} Lemma\,\ref{Lemma2}. 

\m

Although we do not assume $\epsilon>0$ is sufficiently small, we can add this assumption since our horn domains are compact.

\m

To show (\ref{zeroo}), let us first consider the case where the disc contains a singleton canyon $\GC(\ga_*)$.
Since $\epsilon>0$ is sufficiently small, we have  $J^{(e)}(\ga_*)=J^{(e)}(\al_*)$.

Since $d_{gr}(\ga_*)=\infty$, $\ga$ is a multiple root of $f(z,w)$, say of multiplicity $k$, $k\geq2$. Then $Z$ is a multiple factor of $F(Z,W)$ of the same multiplicity.

Consider $\PP(F)$, with edges $E_0$,..., $E_v$. The vertical edge $E_v$ sits on $(k,q)$, for some $q$. 

We must have $e\geq \tan\theta_{v-1}$, for otherwise, as a consequence of Theorem 2.1 in \cite{kuo-par} (or Lemma 3.3 in \cite{kuo-lu}), there would exist a polar $\hat{\ga}$ such that
\bee
O_y(\ga(y)-\hat{\ga}(y))=\tan\theta_{v-1}>e,\quad f(\hat{\ga}(y), y)\not\equiv 0.
\eee
There would then be a contradiction:
\bee
 d_{gr}(\hat{\ga}_*)<\infty, \quad
 \GC(\hat{\ga}_*)\subset \mathcal{D}^{(e)}(\al_*;\epsilon).
\eee

We can therefore restrict our attention to $H_{orn}(\tan\theta_{v-1}, \infty)$. In this horn domain $F_Z$ dominates $F_W$. The argument for proving Lemma\,\ref{Lemma1} can be applied to  prove (\ref{zeroo}).

\m

We can now assume the disc is also disjoint from singleton canyons.

Consider $G\!:=f(Z+\al(W), W)$ and $\PP(G)$. By the above assumption, $G$ must have a dot 
on $Z=0$ or on $Z=1$, for otherwise $\al$ would be a multiple root of $f$, a contradiction.

In either cases, the same argument for proving Lemma\,\ref{Lemma1} can be applied to prove (\ref{zeroo}).

For (\ref{zerooo}), we set $Z=uW^{d+\delta}$, $d\!:=d_{gr}(\ga_*)$, $\delta\!:=e-d,$ and get  $KdS\approx |W|^{2\delta}du\wedge d\bar{u}$, whence (\ref{zerooo}). This completes the proof of Lemma\,\ref{Lemma2}.

\m

Next we introduce two more lemmas. 

With the substitution $z=uw$, we have, as before,
\be\label{kds}
KdS=\frac{2|\Delta_f|^2}{(|f_z|^2+|f_w|^2)^2}\cdot \frac{|w|^4}{|wf_w+uwf_z|^2}\cdot \frac{du\wedge d\bar{u}}{-2\sqrt{-1}}.
\ee

\begin{lem}\label{Lemma4}Take a multiple root $z_i$ of $H_m(z,w)$ in (\ref{mini}), i.e. $m_i\geq 2$. Then
\begin{eqnarray*}
&\lim_{\epsilon \ra 0} \mer_f(\mathcal{D}^{(1)}(z_{i*};\epsilon))=\mer_f(\mathcal{L}^{(1+\delta)}(z_{i*})),&
\end{eqnarray*}
where $\delta >0$ is sufficiently small.

If $z_i$ is a simple root, i.e. $m_i=1$, then
\begin{eqnarray*}
&\lim_{\epsilon\ra 0} \mer_f(\mathcal{D}^{(1)}(z_{i*};\epsilon))=0.&
\end{eqnarray*}
\end{lem}

\begin{proof} For simplicity, we write $H(z,w)\!:=H_m(z,w)$. We can assume $z_i=0$.

If $0$ is a simple root of $H$, then it is not a common root of $H_z$, $H_w$. Hence
\bee
|f_z|^2+|f_w|^2\approx |H_z|^2+|H_w|^2\approx C_\emptyset |w|^{2(m-1)}\quad \text{near $u=0$}.
\eee

We also have
\bee
|\Delta_f|^2\leq C_\emptyset |w|^{6m-8}\cdot |u|^2, \quad |wf_w+uwf_z|^2\approx |w|^{2m}\cdot |u|^{2},
\eee
where the last relation follows from Euler's Theorem. Hence we have
\begin{eqnarray*}
&|KdS|\leq C_\emptyset', \quad \quad \lim_{\epsilon\ra 0}\int_{|u|\leq \epsilon}KdS=0.\end{eqnarray*}

If $0$ is a multiple root of $H$, say of multiplicity $k+1$,  it is a common root of $H_z$, $H_w$, of multiplicity $k$, and vice versa.

In this case, $H_z$ dominates $H_w$ in the horn domain $H_{orn}(1, 1+\delta)$. We can repeat part of the proof of Lemma\,\ref{Lemma1} to complete the proof.
\end{proof}

\begin{lem}\label{Lemma5} Consider the ``punched" plane
\begin{eqnarray*}& \mathcal{P}(\epsilon)\!:=\C_{enriched}-\bigcup _{i=1}^r \mathcal{D}^{(1)}(z_{i*};\epsilon),\quad \epsilon>0.&
\end{eqnarray*}
We have
\begin{eqnarray}\label{punch}
&\lim_{\epsilon\ra 0}\mer_f(\mathcal{P}(\epsilon))=2\pi m(r-1).&
\end{eqnarray}
\end{lem}
\begin{proof} Let us first assume $r=1$. Then we can assume 
\bee
H(z,w)\!:=H_m(z,w)=z^m.
\eee

By Euler's Theorem,
\be\label{euler}
wf_w+uwf_z=mH(uw,w)+\cdots=mu^mw^m+\cdots.
\ee

Of course, we also have
\bee
O(f_z)=m-1<O(f_w), \quad O(f_{zz})=m-2<O(f_{zw}), \;\, etc.
\eee

It follows that
\begin{eqnarray*}
&O(\Delta_f)>3m-4, \quad KdS=o(w)\cdot  \frac{du\wedge d\bar{u}}{-2\sqrt{-1}},&
\end{eqnarray*}
and hence, for any $a>\epsilon$,
\begin{eqnarray*}
&\lim_{w\ra 0}\int_{\epsilon\leq |u|\leq a}KdS=0.&
\end{eqnarray*}

Therefore (\ref{punch}) is true in the case $r=1$. 

\m

\textit{Next we assume} $r\geq 2$.
In this case, the initial form of $\Delta_f$ is
\bee
\Delta_H=\begin{vmatrix}{H_{zz}}&{H_{zw}}&{H_z}\\{H_{wz}}&{H_{ww}}&{H_w}\\{H_z}&{H_w}&{0}\end{vmatrix}.
\eee

 This can be proved as follows. Since $r\geq 2$, we can assume  
 \bee
 H(z,w)=c_kz^kw^{m-k}+\cdots +c_mz^m,\quad c_k\ne 0\ne c_m,\;\,0<k<m.
 \eee
 Then, by a simple computation, we have 
 \bee
\Delta_H=c_k^3mk(m-k)z^{3k-2}w^{3q-2}+\cdots\not\equiv 0,
 \eee
whence $\Delta_H$ is the initial form of $\Delta_f$. 

(It is not difficult to see that  $\Delta_H \equiv 0$ if and only if $H$ has only one factor, i.e. $r=1$.)

 \m

We also have
\bee
wf_w+uwf_z=mH(uw,w)+\cdots=mH(u,1)w^m+\cdots.
\eee

Hence we can rewrite (\ref{kds}) in the form
\bee
KdS=\{\frac{2|\mathcal{R}'(u)|^2}{(1+|\mathcal{R}(u)|^2)^2}+o(W)\}\frac{du\wedge d\bar{u}}{-2i}, \quad \mathcal{R}(u)\!:=\frac{H_z(u,1)}{H_w(u,1)}.
\eee

If $(z-z_iw)^{m_i}$, $m_i\geq 2$, is a factor of $H(z,w)$, then $(z-z_iw)^{m_i-1}$ is a common factor of $H_z$, $H_w$, and vice versa. Hence, having canceled all common factors, we have
\bee
\mathcal{R}(u)=p(u)/q(u), \quad \deg p(u)=r-1\geq \deg q(u), 
\eee
where $p(u)$, $q(u)$  are relatively prime.

\m

Now, the rational function
\bee \mathcal{R}:\, \C \longrightarrow \C, \quad u\mt  \mathcal{R}(u)\!:=U+iV,
\eee 
is an $(r-1)$-fold branch covering, where $U$, $V$ satisfy the Cauchy-Riemann equations.

Take $\delta >0$. Consider the punched disk:
$$P(\delta)\!:=\{(z,w)\in D_{isc}(0;\eta)\,|\,|z-z_iw|\geq \delta|w|,\;1\leq i \leq r\}.$$

\m

An important observation is that the surface $\SU_c\cap P(\delta)$ consists of $m$ sheets, since for each generic $u$ the surface has $m$ distinct intersecting points with the line $z=uw$. 

\m

When the integral of $K$ on $\SU_c\cap  P(\delta)$ is transformed to one over the complex $(U+iV)$-plane, the latter ought to be multiplied by a factor of $m(r-1)$. Thus
$$\lim_{\delta\ra 0}\lim_{W\ra 0}\int_{\SU_c\cap  P(\delta)}KdS=m(r-1)\cdot \lim_{\delta\ra 0}\int_{\C(\delta)} \frac{2dU\wedge dV}{[1+U^2+V^2]^2}= 2\pi m(r-1),$$where $\C(\delta)\!:=\{z\,|\,|z-z_i|\geq \delta, \; 1\leq i\leq r\}$.  
This completes the proof of Lemma\,\ref{Lemma5}.
\end{proof}

\textit{We can now complete the proof of Theorem D.}

Let $\mathcal{L}^{(e)}$, $e>1$, be a given infinitesimal line.

If $\mathcal{L}^{(e)}$ does not contain any polar $\ga_*$, then $\mer_f(\mathcal{L}^{(e)})=0$ by Lemma (\ref{Lemma2}).

Otherwise, we can permute the indices, if necessary, so that
$\{\GC(\ga_{1*}), ..., \GC(\ga_{l*})\}$
are the (distinct) gradient canyons contained in $\mathcal{L}^{(e)}$. 

It can happen that some $\ga_{j*}$ has $d_{gr}(\ga_{j*})=e$. Then, by Theorem B, $\GC(\ga_{j*})$ is necessarily the \textit{only}  canyon contained in $\mathcal{L}^{(e)}$; so that $l=1$ , $\mathcal{L}^{(e)}=\GC(\ga_{1*})$ is minimal, (\ref{DPP}) holds.

Now suppose
\bee
e<d_{gr}(\ga_{j*})\leq \infty, \quad 1\leq j\leq l.
\eee
In this case,
\begin{eqnarray}\label{rec}&
\mer_f(\mathcal{L}^{(e)})=\sum_{j=1}^l\mer_f(\GC(\ga_{j*})),&
\end{eqnarray}
where singleton canyons can be discarded. That is, (\ref{DPP}) is true. 

A proof of (\ref{rec}) is as follows. Take $\al_*\in \mathcal{L}^{(e)}$, which is \textit{not} in one of the above canyons. Let $\al$ be a coordinate of $\al_*$. Take $k$ such that
\bee
O(\al-\ga_k)=\max\{O(\al-\ga_j)\,|\,1 \leq j \leq l\}.
\eee

Now consider
\bee
Q\!:=(q_1, q_2,q_3,q_4)\in \Q^4, \quad q_1\!:=O(\al-\ga_k), \; \,q_2>0,
\eee
and  let
$$\ga^{(Q)}(y)\!:=\ga_k(y)+(q_3+\sqrt{-1}\,q_4)y^{q_1}.$$

If the rational numbers $q_2$, $q_3$ $q_4$ are properly chosen, then
\be\label{unc}
\al_*\in \mathcal{D}^{(q_1)}(\ga^{(Q)}_*;q_2), \quad \mathcal{D}^{(q_1)}(\ga^{(Q)}_*;q_2)\cap[\cup_{1\leq j \leq l}\GC(\ga_{j*})]=\emptyset.
\ee

Of course there are uncountably many $\al_*$, but the number of discs appearing in (\ref{unc}) is countable. We can list them as $\mathcal{D}_n$, $n\in\Z^+$.

Each $\mathcal{D}_n$ is  disjoint from the minimal gradient canyons,  so by Lemma\,\ref{Lemma2},
\bee
\mer_f(\mathcal{D}_n)=0, \quad 1\leq n <\infty,
\eee
and then (\ref{rec}) follows from the identity
\bee\mathcal{L}^{(e)}=[\cup_{1\leq j \leq l}\GC(\ga_{j*})]\cup [\cup_{1\leq n <\infty}\mathcal{D}_n].
\eee

Finally, to prove (\ref{entire}), let $\ga(y)=cy+\cdots$ be given. By Lemma\,\ref{LH}, $1<d_{gr}(\ga_*)\leq \infty$ if and only if $c$ is a multiple root of $H_m(x,1)=0$. Hence (\ref{entire}) follows from Lemma\,\ref{Lemma4} and Lemma\,\ref{Lemma5} and the previous calculation (\ref{rec})
\begin{rem}\label{LRK}
The above  lemmas can be used to show that $\mer_f$ (defined on the enriched discs of positive radius) satisfies the hypothesis of Carath\'{e}odory's Extension Theorem, hence extends to a measure on the $\sigma$-algebra generated by the discs.
\end{rem}

\section{Proof of Theorem E} \label{them}

Let us write (\ref{MWK}) as
\begin{eqnarray*}&\mu_f=\sum_{j=1}^{m-1}\mu_f(\ga_j), \quad \mu_f(\ga_j)\!:=O_y(f(\ga_j(y),y))-1.\end{eqnarray*}
We shall compute each $\mu_f(\ga_j)$.

For convenience, let us re-name and list the roots of $f_z^{(\epsilon)}(z,w)$ and $f_z^{(\delta)}(z,w)$ as
$$\{\ga_1^{(\epsilon)},..., \ga_{m-1}^{(\epsilon)}\} \quad \text{and}\quad \{\ga_1^{(\delta)},..., \ga_{m-1}^{(\delta)}\}$$respectively, in such a way that
\be \label{EEEE}d_{gr}(\ga_j)=O(\ga_j-\ga_j^{(\epsilon)})=O(\ga_j-\ga_j^{(\delta)})=O(\ga_j^{(\epsilon)}-\ga_j^{(\delta)}), \quad 1\leq j \leq m-1.\ee

It is easy to see that
\be\label{ji1}
O(\ga_j-\ga_i)\geq d_{gr}(\ga_j)\implies O(\ga_j^{(\epsilon)}-\ga_i)=d_{gr}(\ga_j),
\ee

and
\be\label{ji2}
O(\ga_j-\ga_i)<d_{gr}(\ga_j)\implies O(\ga_j^{(\epsilon)}-\ga_i)=O(\ga_j-\ga_i).
\ee

\m

Now consider $\PP(F)$, $F(Z,W)\!:=f(Z+\ga_j(W), W)$, as in Fig.\,\ref{fig:npF}.
Recall that $(i,q)$, $i\geq 1$, is a dot of $F$ if and only if $(i-1,q)$ is one of $F_Z$.

The edges of $\PP(F_Z)$ are denoted by $E_i'$ with co-slopes $e_i'$ as in (\ref{eip}), $0\leq i \leq p+1$, where $E_{p+1}' $ is vertical, $(m_i',q_i')$ is the right vertex of $E_i'$.

\m

Amongst the dots of $F_Z$ on $L^*$, let $(m_*,q_*)$ be the lowest one, i.e., $q_*$ is minimal. (As shown in Fig.\ref{fig:npZ}, $F$ has no dots on $L$ lying below $(m_*+1,q_*)$.)
By Theorem\,2.1 in \cite{kuo-par},
\bee \sharp\{i\,|\,O(\ga_j-\ga_i)\geq d_{gr}(\ga_j)\}=m_*.
\eee

To compute $\mu_f(\ga_j)$, let us first assume $d_{gr}(\ga_j)>1$.

Of course $(m_*,q_*)$ is a vertex of $\PP(F_Z)$, say
\bee
(m_*,q_*)=(m_{k+1}', q_{k+1}'), \quad k\leq p.
\eee

Take any $l$, $1\leq l \leq k$. By the same theorem in \cite{kuo-par},
\bee
\#\{i\,|\,O(\ga_j-\ga_i)=e_l'\}=m_l'-m_{l+1}'.
\eee

It follows that (see Fig.\,\ref{fig:npF})
\begin{eqnarray*}&
h-1=m_{k+1}'d_{gr}(\ga_j)+\sum_{l=1}^{k}(m_l'-m_{l+1}')e_l'.&
\end{eqnarray*}

By (\ref{ji1}), (\ref{ji2})
$$m_{k+1}'=\sharp\{i\,|\,O_y(\ga_j^{(\epsilon)}-\ga_i^{(\delta)})= d_{gr}(\ga_j)\},\quad m'_l-m_{l+1}'=\sharp\{i\,|\,O(\ga_j^{(\epsilon)}-\ga_i^{(\delta)})=e_i'\},$$where $1\leq l\leq k$. Hence
\begin{eqnarray}\label{case}&
\mu_f(\ga_j)=h-1=\sum_{i=1}^{m-1}O_y(\ga_j^{(\epsilon)}-\ga_i^{(\delta)}),\end{eqnarray}
and then
\begin{eqnarray*}&\mu_f=\sum_{j=1}^{m-1}\mu_f(\ga_j)=\sum_{j=1}^{m-1}\sum_{i=1}^{m-1}O_y(\ga_j^{(\epsilon)}-\ga_i^{(\delta)})=\mathscr{L}(\mathscr{N}_\varepsilon, \mathscr{N}_\delta).\end{eqnarray*}

Now, assume $d_{gr}(\ga_j)=1$. In this case, $$\tan\theta_{top}=1, \quad \mu_f(\ga_j)=m-1,\quad O_y(\ga_j^{(\epsilon)}-\ga_l^{(\delta)})=1,$$where $1\leq l\leq m-1$.
Hence (\ref{case}) remains true. This completes the proof.
  
\section{Notes}\label{appendix}

(\textbf{I}) We give a proof of (\ref{1b}). (Compare \cite{ness} and \cite{nishimura}.)
First, suppose $f=c$ is a graph:$$f(z,w)-c=w-g(z), \quad g(z)=u(x,y)+\sqrt{-1}\,v(x,y),$$where $g(z)$ is holomorphic,  $z=x+\sqrt{-1}\,y$. For the First Fundamental Form, we have
$$E=G=1+|g'(z)|^2=1+u_x^2+v_x^2,\quad F=0;$$and the (\textit{negative} of the usual) Gaussian curvature (\cite{docar}, p.237) is
\beqn &K=\frac{1}{2E^3}\cdot[E(E_{xx}+E_{yy})-(E_x^2+E_y^2)].& \eeqn

Using the Cauchy-Riemann equations, we then have
\beqn &K=\frac{2}{E^3}[u_{xx}^2+v_{xx}^2]=\frac{2}{E^3}|g^{\prime \prime}(z)|^2=\frac{2\,|\Delta_f|^2}{||\Grad f||^6}.&\eeqn

Now the general case. Near a regular point $(z_0,w_0)$ of $f(z,w)=c$, we can write
$$f(z,w)-c=\mu(z,w)[(w-w_0)-g(z-z_0)], \quad \mu(z_0,w_0)\ne 0\;\;(\mu\;\,\text{a unit}).$$ We then \textit{evaluate the derivatives at} $(z_0,w_0)$: $f_z=-\mu g'$, $f_w=\mu$ and
$$f_{zz}=-2\mu_z g'-\mu g^{\prime \prime}, \; f_{ww}=2\mu_w,\; f_{zw}=-\mu_wg'+\mu_z,$$whence
$\Delta_f(z_0,w_0)=\mu(z_0,w_0)^3g^{\prime \prime}(z_0)$. This completes the proof.

\s

(\textbf{II}) Let us first show how to find the Newton-Puiseux coordinates of a given $\al_*$.
We can apply a unitary transformation, if necessary, so that $T(\al_*)=[0:1]$. 

Take \textit{any} parametrization $\beta(t)=(z(t), w(t))$ of $\al_*$. Then
\begin{eqnarray}\label{unique}&O_t(z(t))>O_t(w(t)), \quad \lim_{t\ra 0} \frac{\|\beta(t)\|}{|w(t)|}=1.&
\end{eqnarray}

Set
$y=w(t)$. Solve $t$ as a fractional power series in $y$, then substitute it into $z(t)$:
\bee
y=w(t)\; \stackrel{solve}{\longrightarrow}t=\tau(y)\stackrel{substitute}{\longrightarrow}\al(y)\!:=z(\tau(y)),
\eee
where $\al(y)$, and the conjugates, are the Newton-Puiseux coordinates of $\al_*$.

Having found $\al(y)\in \F_1$, let us consider $\al_{para}(t)$ in (\ref{para}), and write
\bee
K(\al_{para}(t))=at^{NL}+\cdots,\quad a\ne 0,\;\,L\in \Q,\;\, N\!:=m_{puiseux}(\al),
\eee
where $a$ must be a  positive real number since $K(z,w)$ is. We then have
\be\label{KKK}
\lim_{t\ra 0}\frac{K(\al_{para}(t))}{\|(\al(t^N), t^N)\|^L}=a,
\ee
since $O_y(\al(y))>1$. 
Now let us  replace $t$ by $w(t)^{1/N}$ in (\ref{KKK}), then
$$\lim_{t\ra 0}\frac{K(\beta(t))}{\|\beta(t)\|^L}=\lim_{t\ra 0}\frac{K(\al_{para}(t))}{\|(\al(t^N), t^N))\|^L}=a.
$$
Thus $(a,L)$ is \textit{independent} of the parametrization; $K_*$ is \textit{well-defined}.

\s

(\textbf{III}) The gradient canyons are not topological invariants. For instance, consider  $$F(z,w)=z^3+w^{12}+ tz^2w^5,$$
which is a topologically trivial family. At $t=0$, we have only one double polar with $d=\frac{11}{2}$, hence only one canyon.
For $t\neq 0$, however, there are two disjoint canyons corresponding to the two distinct polars, both having degree  $d=6$.

Gradient canyons are invariants of a stronger notion of equi-singularity, to be studied in another paper.

However, as it follows from our Theorem B,  the minimality of  $\C_{enriched}$ is a topological invariant. Indeed
 $f(z,w)$ has exactly $r$ distinct roots $\zeta_i$ in (\ref{ff})
 if and only if $f$ is topologically equivalent to a homogeneous polynomial germ. This is also equivalent to $f$ having no concentration of curvature at $0\in \C^2$ in the sense of  \cite{kkp1}  (for instance see  Theorem 5.4 in \cite{kkp1}).

\bibliographystyle{amsplain}

\end{document}